\documentclass[a4paper,12pt,english,final]{article}
\usepackage{fancyhdr}
\usepackage[utf8]{inputenc}
\usepackage{geometry}
\usepackage{color}
\usepackage{mathrsfs}
\usepackage{bm}
\usepackage{amsmath,amsthm,amsfonts,amssymb}
\usepackage[english]{babel}
\usepackage{color}
\usepackage{graphicx}
\PassOptionsToPackage{version=3}{mhchem}
\usepackage{mhchem}
\usepackage{wasysym}
\PassOptionsToPackage{normalem}{ulem}
\usepackage{ulem}
\usepackage{relsize}
\usepackage{url}

\makeatletter
\numberwithin{equation}{section}
\numberwithin{figure}{section}
\theoremstyle{plain}
\newtheorem{thm}{\protect\theoremname}[section]
\theoremstyle{plain}
\newtheorem{lem}[thm]{\protect\lemmaname}
\theoremstyle{definition}
\newtheorem{defn}[thm]{\protect\definitionname}
\theoremstyle{remark}
\newtheorem{rem}[thm]{\protect\remarkname}
\theoremstyle{plain}
\newtheorem{assumption}[thm]{\protect\assumptionname}
\theoremstyle{plain}
\newtheorem{prop}[thm]{\protect\propositionname}
\theoremstyle{plain}
\newtheorem{cor}[thm]{\protect\corollaryname}
\theoremstyle{plain}

\definecolor{rot}{rgb}{1.000,0.000,0.000}

\usepackage{tikz}
\usetikzlibrary{arrows.spaced}
\usetikzlibrary{arrows.meta}

\usepackage{tkz-graph}
\GraphInit[vstyle = Shade]
\tikzset{
  LabelStyle/.style = {rectangle, rounded corners, draw,
                        minimum width = 2em, 
                        text = black, font = \bfseries },
  VertexStyle/.style = {circle, draw, 
                                font =  \large\bfseries},
  EdgeStyle/.style = {->, bend left} }
\newcommand{\N}{\mathbb{N}}
\newcommand{\R}{\mathbb{R}}
\newcommand{\eps}{\varepsilon}

\newcommand{\cE}{\mathcal{E}}
\newcommand{\cR}{\mathcal{R}}

\newcommand{\cS}{\mathcal{S}}

\newcommand{\fD}{\mathfrak D}

\renewcommand{\L}{\mathrm{L}}
\newcommand{\W}{\mathrm{W}}

\newcommand{\la}{\langle}
\newcommand{\ra}{\rangle}

\newcommand{\Gammlim}{\xrightarrow{\Gamma}}

\newcommand{\Moscolim}{\xrightarrow{\mathrm{M}}}

\newcommand{\D}{\mathrm{D}}

\newcommand{\eff}{\mathrm{eff}}
\newcommand{\slow}{\mathrm{sl}}
\newcommand{\fast}{\mathrm{fa}}
\newcommand{\scrEfa}{\mathscr{E}_\fast}
\newcommand{\scrMsl}{\mathscr{M}_\slow}
\newcommand{\sfq}{{\mathsf{q}}}
\newcommand{\sfw}{{\mathsf{w}}}
\newcommand{\sfC}{{\mathsf{C}}} 
\newcommand{\sfE}{{\mathsf{E}}}
\newcommand{\sfR}{{\mathsf{R}}}

\newcommand{\Qspace}{\bm{\mathsf{Q}}}
\newcommand{\sfX}{\bm{\mathsf{X}}}

\newcommand{\LB}{\lambda_{\rmB}}
\newcommand{\ul}[1]{\underline{#1}}

\definecolor{rot}{rgb}{1.000,0.000,0.000}

\newcommand{\GamLimE}{\stackrel{\text{$\Gamma_\mathrm{E}$}}{\rightarrow}}
\newcommand{\wGamLimE}{\stackrel{\text{$\Gamma_\mathrm{E}$}}{\rightharpoonup}}
\newcommand{\MoscoLimE}{\stackrel{\text{$\mathrm{M}_\mathrm{E}$}}{\longrightarrow}}

\makeatother

\usepackage{babel}
\providecommand{\assumptionname}{Assumption}
\providecommand{\corollaryname}{Corollary}
\providecommand{\definitionname}{Definition}
\providecommand{\lemmaname}{Lemma}
\providecommand{\propositionname}{Proposition}
\providecommand{\remarkname}{Remark}
\providecommand{\theoremname}{Theorem}

\usepackage[notref,notcite,color]{showkeys}

\let\eps\undefined 
\usepackage{bef_alex_nothm}

\newcommand{\bmR}{{\bm R}}
\newcommand{\mathOP}[1]{\mathop{\mathrm{#1}}}
\makeatletter
\renewcommand*\env@cases[1][1.2]{%
  \let\@ifnextchar\new@ifnextchar
  \left\lbrace
  \def\arraystretch{#1}%
  \array{@{}c@{\ \ }l@{}}%
}
\makeatother

\begin{document}

\title{EDP-convergence for nonlinear fast-slow reaction systems 
         with detailed balance\thanks{Research partially supported by DFG via
           SFB\,1114 (project no.\ 235221301, subproject C05).}}

\author{Alexander Mielke\thanks{Weierstraß-Institut für Angewandte
      Analysis und Stochastik, Mohrenstraße 39, 10117 Berlin
      and Humboldt-Universität zu Berlin, Unter den Linden 6, 10099
      Berlin, Germany, e-mail: \tt{alexander.mielke@wias-berlin.de}}, 
    Mark A. Peletier\thanks{Department of Mathematics and Computer Science and
      Institute for Complex Molecular Systems, TU Eindhoven, 5600 MB Eindhoven,
      The Netherlands, e-mail: \tt{M.A.Peletier@tue.nl}},
and Artur Stephan\thanks{Weierstraß-Institut für Angewandte Analysis und Stochastik, 
        Mohrenstraße 39, 10117 Berlin, e-mail: \tt{artur.stephan@wias-berlin.de}}}

\date{16 October 2020. Revised 9. April 2021}
\maketitle

\lhead{Nonlinear fast-slow reaction systems}
\chead{9.4.2021}
\rhead{A. Mielke, M. Peletier, A. Stephan}

\begin{abstract}
  We consider nonlinear reaction systems satisfying mass-action kinetics with
  slow and fast reactions. It is known that the fast-reaction-rate limit 
  can be described by an ODE with Lagrange
  multipliers and a set of nonlinear constraints that ask the fast reactions to
  be in equilibrium. Our aim is to study the limiting gradient structure which
  is available if the reaction system satisfies the detailed-balance condition.

  The gradient structure on the set of concentration vectors is given in terms
  of the relative Boltzmann entropy and a cosh-type dissipation potential.  We
  show that a limiting or effective gradient structure can be rigorously
  derived via EDP convergence, i.e.\ convergence in the sense of the
  Energy-Dissipation Principle for gradient flows. In general, the effective
  entropy will no longer be of Boltzmann type and the reactions will no longer
  satisfy mass-action kinetics.
\end{abstract}

\section{Introduction}
\label{se:Intro}

The study of nonlinear reaction systems with different time scales has
attracted much attention over the last decades, see e.g.\
\cite{Both03ILRC, KanKur13STSM, WinSch17HMCR, DiLiZi18EGCG,
  MieSte19?CGED, MaaMie20?MCRS} and the references therein.  In this
work we consider the simplest case of fast-slow reaction systems with
mass-action kinetics that have only two time scales, namely $1$ and
$\eps$,
\begin{equation}
  \label{eq:I.RRE}
  \dot c = \bm R_\slow(c) + \frac1\eps \bm R_\fast(c), 
\end{equation}
where $c \in \bfC:= {[0,\infty[}^{i_*}$ denotes the vector of the
concentrations $c_i$ of the $i^\text{th}$ species $X_i$. The typical aim of the
above-mentioned work is to derive the limiting equation for the evolution of
$c$ on the slow time scale, while the fast reactions are in equilibrium. Under
suitable assumptions the limiting equation can be formulated in three
equivalent ways:
\begin{align*}
\text{constrained dynamics:}\quad 
\dot c(t) &= \bm R_\slow(c(t)) + \lambda(t), \quad \lambda(t) \in \Gamma_\fast
\subset \R^{i_*}, \quad \bm R_\fast(c(t))=0,\\[0.4em]
\text{projected dynamics:}\quad
\dot c(t) &=(I{-} \bbP(c(t))) \bm R_\slow(c(t)) , \quad \bm R_\fast(c(0))=0, \\[0.4em]
\text{reduced dynamics:}\quad 
\dot\sfq(t) &= Q_\fast \bm R_\slow(\Psi(\sfq(t))), \quad c(t)=\Psi(\sfq(t)).  
\end{align*} 
We refer to Section \ref{se:EffGS} for a discussion of these
formulations. 

The goal of this work is to revisit the same limit process, but now from the
point of view of variational evolution.  Our starting point is that certain
reaction-rate equations such as \eqref{eq:I.RRE} can be written as a
gradient-flow equation. For a given evolution equation $\dot u=\bm V(u)$ on the
state space $\sfX$ we say that it has a gradient structure, if there exists an
energy functional $\calE:\sfX\to \R$ and a (dual) dissipation potential
$\calR^*:\rmT^*\sfX \to {[0,\infty[}$ such that
$V(u)=\pl_\xi\calR^*(u,{-}\rmD \calE(u))$ for all $u\in \sfX$. This means that
the vector field~$\bm V$ is generated by two scalar-valued functions $\calE$
and $\calR$, which are typically motivated by thermodynamical considerations.
If $\calR^*(u,\cdot)$ is quadratic, viz.\
$\calR^*(u,\xi)= \frac12\langle \xi, \bbK(u)\xi\rangle$, and $\bbK(u)$ is
invertible, then we have a \emph{classical gradient system}, where
$\nabla \calE(u):=\bbK(u)\rmD\calE(u)$ is the gradient, and the quadratic form
$v\mapsto \langle \bbK(u)^{-1}v, v\rangle$ defines a metric tensor.

More general, dissipation potentials are defined via the property that each
$\calR^*(u,\cdot)$ is convex and non-negative with $\calR^*(u,0)=0$; a
corresponding triple $(\sfX, \calE,\calR^*)$ is called a \emph{generalized
  gradient system}.  Each such system generates a unique gradient-flow equation
$\dot u = V(u)=\pl_\xi\calR^*(u,{-}\rmD \calE(u))$. However, for a given
evolution equation $\dot u=\bm V(u)$ there may be zero, one, or several
gradient structures. By the properties of the dissipation potential $\calR^*$,
see Section \ref{su:NotationsGS}, the function $\calE$ is a Liapunov function
decaying along solutions.

There is no general theory concerning the question when a given reaction-rate
equation has a gradient structure.  However, there exists a class of reaction
systems that have a natural gradient structure: these are reaction systems with
mass-action kinetics where the reactions occur in pairs of forward-backward
reactions satisfying the \emph{detailed-balance condition} \eqref{eq:I.DBC}
below. This observation was highlighted in \cite[Sec.\,3.1]{Miel11GSRD}
but was observed and used implicitly earlier in
\cite[Eqs.\,(103)+(113)]{OttGrm97DTCF2} and \cite[Sec.\,VII]{Yong08ICPD}.  A
different gradient structure already occurs in \cite[Eqn.\,(69)]{Grme10MENT}
and has its origin in the thermodynamic considerations in \cite{Marc15CECP}
from 1915. The latter structure, which we will call the
\emph{cosh-type gradient structure} as in \cite{MieSte19?CGED}, was
mathematically derived in \cite{MiPeRe14RGFL, MPPR17NETP} from microscopic
chemical master equations via a large-deviation principle.

To be specific, we assume that
the species $X_i$, $i\in I:=\{1,\ldots,i_*\}$ undergo $r_*$
forward-backward reactions according to the mass-action kinetics 
\[
 \alpha_{1}^{r}X_{1} +\cdots+ \alpha_{i_*}^{r}X_{i_*}\quad
 \rightleftharpoons \quad
  \beta_{1}^{r}X_{1} + \cdots + \beta_{i_*}^{r}X_{i_*},
\] 
where $\alpha^r =(\alpha_i^r)_{i\in I}$ and $\beta^r
=(\beta_i^r)_{i\in I}$ are the stoichiometric vectors in
$\N_0^{i_*}$. The reaction-rate equation \eqref{eq:I.RRE} takes the form 
\begin{equation}
  \label{eq:I.fsRREmak}
\dot c = -\sum_{r=1}^{r_*} \Big(k^\mafo{fw}_r c^{\alpha^r} - k^\mafo{bw}_r
c^{\beta^r}\Big)\big( \alpha^r{-}\beta^r), \quad \text{where }
c^\alpha = c_1^{\alpha_1}\cdots c_{i_*}^{\alpha_{i_*}}.
\end{equation}
The \emph{detailed-balance condition} asks for the existence of a
positive concentration vector $c_*=(c_i^*)_{i\in I} \in \bfC_+:=
{]0,\infty[}^{i_*}$ such that all $r_*$ reactions are in
\begin{equation}
  \label{eq:I.DBC}
  \exists\, c_*=(c_i^*)_{i\in I} \in \bfC_+\ \forall\, r\in
  R:=\{1,\ldots,r_*\}:\quad k^\mafo{fw}_r c_*^{\alpha^r} = k^\mafo{bw}_r
c^{\beta^r}_*\,. 
\end{equation}
This means that there is one equilibrium vector $c_*$ such that all
reaction pairs are in equilibrium simultaneously. The reaction strength of a
reaction pair can then be measured by $\wh\kappa_r = k^\mafo{fw}_r
c_*^{\alpha^r}/\delta_r^*= k^\mafo{bw}_r c^{\beta^r}_*/\delta_r^*$ where 
$\delta^*_r = \big( c_*^{\alpha^r} c^{\beta^r}_*\big)^{1/2}$. 

The set of reaction pairs $R$ will be decomposed into
slow and fast reactions, namely $R= R_\slow \,\dot\cup\,
R_\fast$ and by assuming  $\wh\kappa_r = \kappa_r$ for $r \in R_\slow$
and $\wh\kappa_r=\kappa_r/\eps$ for $r\in R_\fast$, where $\kappa_r$
are fixed numbers. Hence, slow reactions occur on the time scale $O(1)$,
whereas fast reactions occur on the time scale $O(\eps)$. 
The fast-slow reaction-rate equation now reads
\begin{equation}
  \label{eq:I.fsRREsym}
  \dot c = \bm R_\slow(c) + \frac1\eps \bm R_\fast(c) \quad \text{with
    } \bm R_\mathrm{xy}(c) := - \sum_{r\in R_\mathrm{xy}} \kappa_r \,\delta_r^* \bigg(
  \frac{c^{\alpha^r}}{c_*^{\alpha_r}} -
  \frac{c^{\beta^r}}{c_*^{\beta_r}}\bigg) ( \alpha^r - \beta^r).
\end{equation}
Throughout this work, we assume that the equilibrium vector $c_*$ does not
depend on $\eps$.

The cosh-type gradient structure for is now defined in terms of a gradient
system $(\bfC,\cE,\cR^*_\eps)$, where the energy functional is given in terms
of the relative Boltzmann entropy 
\[
\cE(c)= \sum_{i\in I} c^*_i \LB(c_i/c^*_i), \quad \text{where } 
    \LB(\rho):= \rho \log \rho - \rho +1,
\]    
and the dual dissipation potential $\cR^*_\eps$ in the form 
\[
\cR^*_\eps(c,\xi) = \cR^*_\slow(c,\xi) + \frac1\eps \cR^*_\fast(c,\xi)
\quad \text{with }
\cR^*_\mathrm{xy}(c,\xi) = \sum_{r\in R_\mathrm{xy}} \kappa_r 
\big( c^{\alpha^r}c^{\beta^r}\big)^{1/2}\:
\sfC^*\big((\alpha^r{-}\beta^r) \cdot \xi \big),   
\]
where $\sfC^*(\zeta)= 4\cosh(\zeta/2) - 4$ involves the ``cosh
structure''. There is now a special and absolutely non-trivial interaction
between the Bolzmann entropy, the mass-action law, and the cosh-type gradient
structure, which relies on the fact that $\rmD\calE(c)$ is the vector 
of logarithms, namely $\xi=\rmD\calE(c)= \big(\log (c_i/c_*^i) \big)_i
$. Multiplying this by the stoichimetric vectors 
$\alpha^r-\beta^r$ and using the logarithm rules we obtain 
\[
(\alpha^r{-}\beta^r) \cdot \xi = (\alpha^r{-}\beta^r) \cdot\rmD\calE(c) =
\log\Big( \frac{c^{\alpha^r}} {c_*^{\alpha^r}}\Big) - \log\Big(\frac{ c^{\beta^r}}
{c_*^{\beta^r}}  \Big)  .
\] 
For evaluating $\pl_\xi \calR^*_\mathrm{xy}(c,\xi)$ these terms are inserted
into
$(\sfC^*)'(\zeta) = 2\sinh(\zeta/2) = \big(\ee^{\zeta}\big)^{1/2} -
\big(\ee^{-\zeta}\big)^{1/2} $, which leads to a cancellation of the logarithms
and the desired monomials appear after exploiting the square roots
$\big( c^{\alpha^r}c^{\beta^r}\big)^{1/2}$ in $\calR^*_\mathrm{xy}$ and in
$\delta_r^*=\big(c_*^{\alpha^r}c_*^{\beta^r}\big)^{1/2} $. Thus, the
fast-slow reaction-rate equation \eqref{eq:I.fsRREsym} indeed takes the form of the
gradient-flow equation
\[
\dot c(t) = \pl_\xi\cR^*_\eps\big( c(t), {-}\D \cE(c(t))\big).
\] 
In fact, there are many other gradient structures for
\eqref{eq:I.fsRREsym}, see Remark
\ref{re:SeveralGS}; however the cosh-type gradient structure is
special in several aspects: (i) it can be derived via large-deviation
principles \cite{MiPeRe14RGFL, MPPR17NETP}, (ii) the dual dissipation
potential $\cR^*_\eps$ is independent of $c_*$, and (iii) it is stable
under general limiting processes, see \cite[Sec.\,3.3]{LMPR17MOGG}. 
The property (ii), also called  \emph{tilt invariance} below, 
will be especially important for us.

The main goal of this paper is to construct the effective gradient
system $(\bfC,\cE_\eff,\cR^*_\eff)$ for the given family
$(\bfC,\cE,\cR_\eps^*)$ in the limit $\eps \to 0^+$. Here we use the
notion of \emph{convergence of gradient system in the sense of the
energy-dissipation principle}, shortly called
\emph{EDP-convergence}. This convergence notion was introduced in
\cite{DoFrMi19GSWE} and further developed in \cite{MiMoPe18?EFED,
    FreLie19?EDTS, MieSte19?CGED} and is based on the dissipation
  functionals
\[
\fD^\eta_\eps(c):= \int_0^T \! \big\{ \cR_\eps(c,\dot c) + \cR^*_\eps ( c,
\eta{-}\D\cE(c)) \big\} \dd t,
\] 
which are defined for curves $c \in \L^1([0,T];\bfC)$. Here $\calR_\eps$ is the
primal dissipation potential conjugated to $\calR_\eps^*$, see
\eqref{eq:LegFenTrafo}.  The
notion of \emph{EDP-convergence with tilting} now asks that the two
$\Gamma$-convergences $\cE_\eps \Gammlim \cE_\eff$ and
$\fD^\eta_\eps \Gammlim \fD_0^\eta$ (in suitable topologies) and that for
all $\eta $ the limit $\fD^\eta_0$ has the form
$\fD^\eta_0(c)=\int_0^T \!\big\{ \cR_\eff(c,\dot c) + \cR^*_\eff
(c,\eta{-}\D\cE(c))\big\} \dd t $; see Section \ref{su:Def.EDPcvg}
for the exact definitions of $\Gamma$-convergence and EDP-convergence. 

Our main result is Theorem \ref{thm:tiltEDPcvg}, which asserts
EDP-convergence with tilting and leading to the effective gradient system
$(\bfC, \calE_\eff,\calR_\eff)$  with 
\[
\cE_\eff=\cE \qquad \text{and} \qquad \cR_\eff^*(c,\xi) =
\cR^*_\slow(c,\xi) + \chi_{\Gamma_\fast^\perp}(\xi),
\]
where $\Gamma_\fast=\mathOP{span}\bigset{\alpha^r{-}\beta^r}{ r\in
  R_\fast}$, $\Gamma_\fast^\perp:=\bigset{\xi \in \R^{i_*}}{
  \forall\, \gamma\in \Gamma_\fast:\ \gamma\cdot\xi =0}$ and $\chi_A$ is the characteristic function of convex analysis taking 0 on $A$ and $\infty$ otherwise. 
The proof relies on three important observations:

(1) Tilting of the relative Boltzmann entropy $\cE$ by $\eta$,  i.e.\
replacing $\calE(c)$ by $\calE(c)-\eta {\cdot} c$, is equivalent to
changing the underlying equilibrium $c_*$ to 
$c_*^\eta:=(\ee^{\eta_i}c_i^*)_{i\in I}$ (see
\eqref{eq:TiltRelBoltz}), and $\cR^*_\eps$ is independent of
$c_*^\eta$.


(2) The factor $1/\eps$ in front of the fast reaction and in front of
$\calR^*_\fast$ allows for fast changes of $c$ in the corresponding directions.
These directions are given by the fast stoichiometric subspace $\Gamma_\fast$.
Defining an operator $Q_\fast : \R^{i_*} \to \R^{m_\fast}$ such that
$\mathOP{ker} Q_\fast = \Gamma_\fast$ and
$\mathOP{im} Q_\fast^\top= \Gamma_\fast^\perp$, a dissipation bound
$\fD_\eps^\eta(c^\eps)\leq M_\mafo{diss} < \infty$ does provide a uniform bound
on $Q_\fast c^\eps$ in $\W^{1,1}([0,T];\R^{m_\fast}) $, but not on $c^\eps$ as
a whole. The reason is that the blow-up of the dual dissipation potential
$\calR_\eps^*(c,\cdot)$ along $\Gamma_\fast^\perp$ is mirrored by a
degeneration of the primal dissipation potential $\calR_\eps$ along
$\Gamma_\fast$, i.e.\ for all $v \in \Gamma_\fast$ we have
$\calR_\eps (c,v)\to 0 $ for $\eps\to 0$.

(3) Since $\calR_\eps\geq 0$, a dissipation bound
$ \fD_\eps(c^\eps)\leq M_\mafo{diss} $ trivially implies the bound
$\int_0^T \frac1\eps \cR^*_\fast(c^\eps, {-}\D\cE(c^\eps))\dd t \leq
M_\mafo{diss}$. Analyzing the function
$c \mapsto \calR_\fast(c,{-}\rmD\calE(c))\geq 0$ for our mass-action reaction
kinetics shows that its $0$-set is exactly given by the set of equilibria of the
fast equation, namely $\scrEfa:=\bigset{c\in \bfC}{\bm R_\fast(c)=0}$. Hence, a
dissipation bound $\fD_\eps(c^\eps)\leq M_\mafo{diss} <\infty$ forces the
family $\big(c^\eps\big)_\eps$ to converge towards $\scrEfa$, i.e.\ in the
limit the fast reactions have to be in equilibrium for almost all times. 
\smallskip

Our analysis is based on the important assumption that the fast reaction
system $c'(\tau) = \bm R_\fast(c(\tau))$ has a unique equilibrium in each
flow-invariant subset $\bfC^\fast_\sfq:=\bigset{c\in \bfC}{ Q_\fast c =
  \sfq}$. This equilibrium is obtained as minimizer of $\cE$ on
$\bfC^\fast_\sfq$ and is denoted by $\Psi(\sfq)$. Thus, the \emph{unique
  fast-equilibrium condition} (UFEC) reads
\[
  \scrMsl:=\bigset{\Psi(\sfq)}{ \sfq \in Q_\fast \bfC} \ \overset{!!}= \
  \scrEfa:=\bigset{c\in \bfC}{\bm R_\fast(c)=0},
\]
which means that there are no ``additional boundary equilibria'', see
Assumption \ref{as:UniqueFastEquil}. 

The main difficulty is to show that the information in points (2) and (3) is
enough to obtain the compactness necessary for deriving liminf estimate for the
$\Gamma$-convergence $\fD_\eps \Gammlim \fD_0$ for the non-convex functionals
$\fD_\eps$. On the local level, one sees that (2) provides partial control of
the temporal oscillations of $\dot c^\eps$ via the bound on
$Q_\fast \dot c^\eps$ in $\L^1([0,T];\R^{m_\fast})$, whereas (3) provides
strong convergence towards the nonlinear manifold $\scrMsl$, which is
locally defined via $\D \cE(c) \in \Gamma_\fast^\perp$ (see Lemma
\ref{le:DEGamFast}). In summary, we are able to show that
$\fD_\eps(c^\eps)\leq M_\mafo{diss}<\infty$ implies that there exists a
subsequence such that $c^{\eps_n} \to \wt c$ in $\L^1([0,T];\bfC)$ and
$Q_\fast c^{\eps_n} \to \sfq$ uniformly in $\rmC^0([0,T];\R^{m_\fast})$, where
$\wt c(t)=\Psi(\sfq(t))$ with $\sfq \in \W^{1,1}([0,T];\R^{m_\fast})$.

As a corollary we obtain that the limiting evolution lies in $\scrMsl$ and is
governed by the reduced (or coarse grained) equation
$\dot\sfq = Q_\fast \bm R_\slow (\Psi(\sfq))$ described by the slow
variables $\sfq \in Q_\fast \bfC$ and a natural gradient structure $(Q_\fast
\bfC, \sfE,\sfR)$. Even on the level of the limiting
equations our result goes beyond those in \cite{Both03ILRC, DiLiZi18EGCG},
since we do not assume that solutions are strictly positive or that the
stoichiometric vectors $\gamma^r = \alpha^r{-}\beta^r$, $r\in R_\fast$, are
linearly independent.

For illustration, we close the introduction by a simple example involving
$i_*=5$ species and one fast and one slow reaction (see Section~\ref{su:Exa.EffGS} for the details under slightly more general conditions):
\[
 X_{1}+X_{2}\overset{\text{fast}}\rightleftharpoons X_{3}
  \qquad  \text{and}\qquad
  X_{3}+X_{4}\overset{\text{slow}}\rightleftharpoons X_{5}, 
\]
which gives rise to the stoichiometric vectors $\bfgamma^\fast=(1,1,-1,0,0)^\top$
and $\bfgamma^\slow=(0,0,1,1,-1)^\top$. Assuming the detailed-balance condition
with respect to the steady state $c_*=(1,1,1,1,1)^\top$, the  
reaction-rate equation takes the form
\[
  \dot{c}=-\frac{\kappa^\fast}\eps \big(c_{1}c_{2}-c_{3}\big)\bfgamma^\fast
  - \kappa^\slow \big(c_{3}c_{4} -c_{5}\big) \bfgamma^\slow \,.
\]
The limiting reaction system can be described by the slow variables
$\sfq {=} (c_1{+}c_3,c_2{+}c_3, c_4,c_5)^\top$ and reads
\[
\dot\sfq  = Q_\fast \bm{R}_\slow(\Psi(\sfq)) = -\kappa^\slow
\bigl(a(q_1,q_2)q_3 -q_4\bigr) \bfgamma^\slow,
\]
where the slow manifold takes the form
$\Psi(\sfq)=\bigl(q_1,q_2,a(q_1,q_2),q_3,q_4\bigr)$ and the reduced entropy is
$\sfE(\sfq)=\calE(\Psi(\sfq))$.

\section{Modeling of reaction systems}
\label{se:ModelRS}

We first introduce the classical notation for reaction systems
with reaction kinetics according to the mass-action law. After
briefly recalling our notation for \textit{gradient systems}, 
we show that based on the condition of detailed balance, the
reaction-rate equation is the gradient-flow equation for a suitable
gradient system. Next we introduce our class of fast-slow systems,
and finally we present a small, but nontrivial example in $\R^3$.

\subsection{Mass action law and stoichiometric subspaces}
\label{su:MassAction}

We consider $i_{*}\in\mathbb{N}$ species $X_{i}$ reacting with each
other by $r_{*}\in\mathbb{N}$ reactions. The set of species is denoted
by $I=\left\{ 1,\dots,i_{*}\right\} $, the set of reactions by
$R=\left\{ 1,\dots,r_{*}\right\} $, and the $r_{*}$ chemical reactions
are given by
\[
  \forall\, r\in R:\quad 
  \sum_{i=1}^{i_{*}}\alpha_{i}^{r}X_{i}\ \ \rightleftharpoons \ \ 
  \sum_{i=1}^{i_{*}}\beta_{i}^{r}X_{i},
\]
where the stoichiometric vectors
$\mathbf{\alpha}^{r},\mathbb{\beta}^{r}\in\N_{0}^{i_{*}}$ contain the
stoichiometric coefficients. The concentration $c_{i}$ of species
$X_{i}$ is nonnegative, the space of concentrations is denoted by
\[
 \bfC =[0,\infty[^{i_{*}}\ \subset\ \R^{i_{*}}\ ,
\]
which is the nonnegative cone of $\R^{i_{*}}$. Moreover, we introduce
$ \bfC _{+}:=\mathOP{int} \bfC  ={]0,\infty[}^{i_{*}}$, the interior
of the set of concentrations.

The mass-action law for reaction kinetics assumes that the
forward and backward reaction fluxes are proportional to the product
of the densities of the species,  i.e.
$j(c)_{r}=- k_{r}^{\mathrm{fw}}c^{\alpha^{r}}
+k_{r}^{\mathrm{bw}}c^{\beta^{r}}$, where for stoichiometric vectors 
$\delta \in \N_0^{i_+}$ the monomials $c^\delta$ are given by
$\prod_{i=1}^{i_*} c_i^{\delta_i}$. The  reaction-rate equation (RRE) of the
concentrations $c\in \bfC $  takes the form 
\begin{equation}
  \label{eq:RRE}
  \dot{c}=\bm{R}(c)=-\sum_{r=1}^{r_{*}}\left(k_{r}^{\mathrm{fw}} 
 c^{\alpha^{r}}-k_{r}^{\mathrm{bw}}c^{\beta^{r}}\right) 
 \left(\alpha^{r}-\beta^{r}\right),
\end{equation}
with given forward and backward reaction rates
$k_{r}^{\mathrm{fw}},k_{r}^{\mathrm{bw}}>0$.

For each of the $r$ reactions we introduce the stoichiometric 
vector $\gamma^{r}:=\alpha^{r}-\beta^{r}\in \Z^{i_{*}}$.  The
span of all vectors $\gamma^{r}$ is the \emph{stoichiometric
  subspace} $\Gamma\subset\R^{i_{*}}$, i.e.\
$\Gamma:=\mathOP{span}\bigset{ \gamma^{r} }{ r\in R } $. We do not
assume any properties of the stoichiometric vectors, in particular
they are not assumed to be linearly independent.

%
%
Conservation directions are vectors $q\in\R^{i_{*}}$ such
that $q\in\Gamma^{\perp}$ (also written $q\perp\Gamma$), where the
annihilator $\Gamma^{\perp}$ is defined as
$\Gamma^{\perp}=\bigset{ q\in\R^{i_{*}} }{ \forall\,\gamma\in\Gamma: \
  q\cdot\gamma=0 } $. By construction we have
$\bmR(c) \in \Gamma$, thus for all solutions $t \mapsto c(t)$ of the
RRE \eqref{eq:RRE}, the value of $q\cdot c(t)$ is constant, i.e.\
$q\cdot c$ is a conserved quantity for \eqref{eq:RRE}. 
Fixing a basis $\left\{ q_{1},\dots,q_{m}\right\} $ of
$\Gamma^\perp$, we introduce a matrix $Q\in\R^{m\times i_{*}}$ by
defining its adjoint $Q^\top=\left(q_{1},\dots,q_{m}\right)$. By
construction, $Q^\top:\R^m \to \R^{i_*}$ is injective,
$Q: \R^{i_*} \to \R^m$ is surjective, and
$\mathOP{ker} Q=\Gamma$. The image of the nonnegative cone
$ \bfC $ under $Q$ is denoted by $\mathscr{Q} $, i.e.\
$Q: \bfC  \to \mathscr{Q}\subset\R^{m}$.  Fixing a vector
$q\in\mathscr{Q}$, we define the stoichiometric subsets 
\[
 \bfC _{q}:=\bigset{ c\in \bfC }{ Qc=q } \, .
\]
They provide a decomposition
$\bfC={\larger\cup}_{q\in\mathscr{Q}} \bfC _{q}$ into affine sets that are
invariant under the flow of the RRE \eqref{eq:RRE}.\smallskip

\textbf{Notation:} In the whole paper we consider all vectors as column
vectors. In particular $\D\cE(c)\in X^*$ is also a column vector although it is
an element of the dual space and might be understood as a covector.

\subsection{Notations for gradient systems}
\label{su:NotationsGS}

Following \cite{Miel11GSRD, Miel16EGCG}, we call a triple $( \sfX ,\cE,\cR)$
a (generalized) \emph{gradient system} (GS) if 
\begin{enumerate}\itemsep-0.2em
\item the state space $ \sfX $ is a closed and convex
  subspace of a Banach space $X$,  
\item $\cE: \sfX  \to \R_{\infty}:=\R\cup\{\infty\}$ is a
  sufficiently smooth functional (such as a free energy,
  a relative entropy, or a negative entropy, etc.),
\item $\cR: \sfX \times X \to \R_{\infty}$ is a dissipation
  potential, which means that for any $u\in \sfX $ the functional
  $\cR(u,\cdot): X  \to \R_{\infty}$ is lower
  semicontinuous, nonnegative and convex, and satisfies $\cR(u,0)=0$.
\end{enumerate}
The dynamics of a gradient system is given by the associated
\emph{gradient-flow equation} that can be formulated in three
different, but equivalent ways: as an equation in $X$, in $\R$, or in
$X^{*}$ (the dual Banach space of $X$), respectively:
\begin{subequations}
\label{eq:GFE.gen}
\begin{align}
\text{(I)\ \ \  \textbf{Force balance in }}X^{*}:&&
  0\in\partial_{\dot{u}}\cR(u,\dot{u})+\D\cE(u) && \in X^{*},
\\  
\text{(II) \  \textbf{ Power balance in }}\R:&& 
 \cR(u,\dot{u})+\cR^{*}(u,-\D\cE(u))=-\la\D\cE(u),\dot{u}\ra&& \in \R,
\\ 
\text{(III) \textbf{ Rate equation in }}X:  
  &&\dot{u}\in\partial_{\xi}\cR^{*}(u,-\D\cE(u))&& \subset X. 
\end{align}
\end{subequations}
Here, $\cR^*$ is the \emph{dual dissipation potential} obtained by
the Legendre-Fenchel transform 
\begin{equation}
  \label{eq:LegFenTrafo}
\cR^*(u,\xi) :=\sup_{ v\in X}\big\{\langle \xi,  v\rangle -\cR(u,v)\}. 
\end{equation}
In general, the partial derivatives $\pl_{\dot u}\cR(u,\dot u)$ and
$\pl_\xi \cR^*(u,\xi)$ are the possibly set-valued convex subdifferentials.

For a given evolution equation $\dot u = \bm V(u)$ we say that it has
a \emph{gradient structure} if there exists a gradient system
$(\sfX,\cE,\cR)$ such that the evolution equation is the gradient-flow
equation for this gradient system, namely $\bm V(u)=\pl_\xi
\cR^*(u,{-}\D\cE(u))$. We emphasize that a given evolution equation may
have none or many gradient structures; see Remark \ref{re:SeveralGS}
for the case of our nonlinear reaction systems. 

Integrating the power balance (II) in time over $[0,T]$ and using
the chain rule for the time-derivative of $t\mapsto\cE(u(t))$, we
obtain another equivalent formulation of the dynamics of the gradient
system, which is called \textit{Energy-Dissipation-Balance}: 
\begin{align}
\label{eq:EDB.gen}
\mathrm{(EDB)} \qquad \cE(u(T))+\int_{0}^T \!\!\big\{\cR(u,\dot{u})+
   \cR^{*}(u,-\D\cE(u))\big\} \dd t=\cE(u(0)).
\end{align}
This gives rise to the
\textit{dissipation functional} 
\begin{align*}
\fD(u):=\int_{0}^T  \! \big\{\cR(u,\dot{u})+\cR^{*}(u,-\D\cE(u))\big\}\dd t,
\end{align*}
which is now defined on trajectories $u:[0,T]\mapsto \sfX$. 

The following \emph{Energy-Dissipation Principle} (EDP) states that,
under natural technical conditions, 
solving the EDB \eqref{eq:EDB.gen} is equivalent to solving any of
the three versions of the gradient-flow
equation \eqref{eq:GFE.gen}.

\begin{thm}[{Energy-dissipation principle, cf.\
  \cite[Prop.\,1.4.1]{AmGiSa05GFMS} or
  \cite[Th.\,3.2]{Miel16EGCG}}]\label{th:EDP}%
\mbox{}
Assume that $\sfX$ is a closed convex subset of $X=\R^{i_*}$, that $\cE\in
\rmC^1(\sfX,\R)$, and that the dissipation potential $\cR(u,\cdot)$ is
superlinear uniformly in $u\in \sfX$. Then, a function $u \in
\rmW^{1,1}([0,T];\R^{i_*})$ is a solution of the gradient-flow equation 
\eqref{eq:GFE.gen} if and only if $u$ solves the EDB \eqref{eq:EDB.gen}.
\end{thm}

\subsection{The detailed balance condition induces gradient structures}
\label{su:DBCiGS}

Already in Section \ref{su:MassAction}, we have assumed that each
reaction occurs in both forward and backward directions. Such reaction
systems are called \emph{weakly reversible}. A much stronger assumption is the
so-called \emph{detailed-balance condition} which states that there is
a strictly positive state $c_*=(c^*_i)\in \bfC_+$ in which
all reactions are in equilibrium, i.e.\ $j_r(c_*)=0$ for all $r$:
\begin{equation}
  \label{eq:DBC1}
\text{(DBC)}\qquad \exists\, c_* \in \bfC_+\ \forall\, r\in \R:\quad  
k_{r}^{\mathrm{fw}}c_{*}^{\alpha^{r}}=k_{r}^{\mathrm{bw}}c_{*}^{\beta^{r}}.
\end{equation}
Under this assumption, one can rewrite the RRE \eqref{eq:RRE} in the symmetric form
\begin{equation}
  \label{eq:RRE.sym}
\begin{aligned} & \dot c = \bmR(c)=- \sum_{r=1}^{r_*} \wh\kappa_r \,\delta^*_r
  \Big( \,\frac{c^{\alpha^r}}{c_*^{\alpha^r}}  -
   \frac{c^{\beta^r}}{c_*^{\beta^r}}\,\Big)
  \,\big(\alpha^r- \beta^r\big) 
\\
&\text{with } \ \delta^*_r = \big(c_*^{\alpha^r}c_*^{\beta^r}\big)^{1/2}  \ \text{
  and } \  \wh\kappa_r:=k_{r}^{\mathrm{fw}}
  c_*^{\alpha^r}/\delta^*_r   =k_{r}^{\mathrm{bw}}c_*^{\beta^r}/\delta^*_r\,.
\end{aligned} 
\end{equation}
Subsequently, we will use the notion of a \emph{reaction system satisfying
the detailed-balance condition}, or shortly a detailed-balance
reaction system. 

\begin{defn}[Detailed-balance reaction systems (DBRS)]\label{def:DetBalReaSys}
  For $i_*,r_*\in \N$ consider the stoichiometric matrices $A = \big( \alpha^r_i
  \big) \in \N_0^{i_*\ti r_*}$ and $ B = \big( \beta^r_i \big) \in
  \N_0^{i_*\ti r_*}$ and the vectors $c_*=(c^*_i) \in
  {]0,\infty[}^{i_*}$ and $\wh \kappa=(\wh \kappa_r)\in
  {]0,\infty[}^{r_*}$. Then, the quadruple $(A,B,c_*,\wh\kappa)$ is
  called a \emph{detailed-balance
reaction systems} with $i_*$ species and $r_*$ reactions. The
associated RRE is given by \eqref{eq:RRE.sym}.  
\end{defn} 

It was observed in \cite{Miel11GSRD} (but see also
\cite[Eqn.\,(103)+(113)]{OttGrm97DTCF2} and
\cite[Sec.\,VII]{Yong08ICPD} for earlier, but implicit statements)
that RREs in this form have a gradient structure. Here we will use the
gradient structure derived in \cite{MPPR17NETP} by a large-deviation
principle from a microscopic Markov process. In Remark
\ref{re:SeveralGS} we will shortly comment on other possible gradient
structures.

With $\bfC$ as above we define the energy as the relative Boltzmann
entropy 
\begin{subequations}
\label{eq:RRE.GS}
\begin{align}
\label{eq:RRE.GS.a}
\cE:\left\{ \ba{cccc}  \bfC  &\to& \R,\\
c &\mapsto & \sum_{i=1}^{i_{*}}c_{i}^{*} \LB 
(c_{i}/c_{i}^{*}), \ea\right. 
\quad \text{with } 
 \LB (r)=\begin{cases}
   r\log r-r+1 & \text{for }r>0, \\
             1 & \text{for }r=0,\\
        \infty &\text{for } r<0. \end{cases}
\end{align}
The dissipation functional $\cR$ will be defined by specifying the
dual dissipation potential $\cR^*$ of ``cosh-type'' as
\begin{align}
\label{eq:RRE.GS.b}
\cR^{*}: \left\{ \ba{ccc} \!\! \bfC \times\R^{i_{*}}\!\! &\to& \R,\\  
(c,\xi)&\mapsto &\!\!\ds\sum_{r=1}^{r_*} \wh\kappa_{r} \sqrt{c^{\alpha^{r}}
  c^{\beta^{r}}}\:
 \sfC^{*}\big(\,(\alpha^{r}{-}\beta^{r})\cdot\xi\,\big),
\ea\right. 
\ \ 
\text{ with }  \sfC^{*}(\zeta)=4\cosh\left(\zeta/2\right)-4\, .
\end{align}
\end{subequations}
We will often use the following formulas for $ \sfC^*$:
\begin{align}
\label{eq:C.formulas}
\begin{aligned}
\text{(a) }& 
  \sfC^{*}(\log p{-}\log q)  = 2 \frac{\left(\sqrt p-\sqrt q 
 \right)^{2}}{\sqrt{pq}} , \\
\text{(b) }&  
\big( \sfC^*\big)'(\zeta)  =\ee^{\,\zeta/2}- \ee^{-\zeta/2}, 
\qquad
\text{(c) } \big( \sfC^{*}\big)'(\log p{-}\log
q)=\frac{p-q}{\sqrt{pq}}.
\end{aligned}
\end{align}

The following result is also easily checked by direct calculations
using~\eqref{eq:C.formulas}(b) and the logarithm rules 
\begin{equation}
  \label{eq:Logarithm}
  \alpha^{r}\D\cE(c)=\log(c^{\alpha^{r}})-\log(c_{*}^{\alpha^{r}}) =
\log \big(c^{\alpha^{r}}/c_{*}^{\alpha^{r}} ).
\end{equation}
This identity also follows as a special case of Remark \ref{re:SeveralGS}. 
The primal dissipation potential $\cR$ is given by the
Legendre-Fenchel transformation:  
\begin{equation}
  \label{eq:def.cR}
  \cR(u,v)=\sup\bigset{ \xi\cdot v -\cR^*(c,\xi)}{ \xi \in \R^{i_*}}. 
\end{equation}

\begin{prop}[{Gradient structure, \cite[Thm.\,3.6]{MPPR17NETP}}] 
\label{pr:GS.cosh}
The RRE \eqref{eq:RRE.sym} is the gradient-flow equation associated with the
cosh-type gradient system $(\bfC,\cE,\cR)$ with $\cE$ and $\cR$ given in
\eqref{eq:RRE.GS}, where $\calR$ and $\calR^*$ are related by
Legendre-Fenchel transform, see \eqref{eq:LegFenTrafo}.
\end{prop}

An important property of this gradient structure, which is not shared with the
ones discussed in Remark \ref{re:SeveralGS} below, is that the dissipation
potential $\cR^*$ does not depend on the equilibrium state $c_*$, see also
Section \ref{su:Example} for an example. This property might seem to
be an artifact of our special choice of the definition of $\wh\kappa_r$ in
terms of $c_*$; however, it is an intrinsic property that will be
even more relevant when we use ``tilting'' in our main result Theorem
\ref{thm:tiltEDPcvg}, which states the ``EDP-convergence with tilting''. In
\cite[Prop.\,4.1]{MieSte19?CGED} it was shown that this tilt-invariance is a
special property of the cosh-gradient structure; see also Remark
\ref{re:SeveralGS}.

Moreover, we have identified $c_*$ as a ``static'' property of
the RRE \eqref{eq:RRE.sym}, whereas the stoichiometric matrices
$A,B\in \N_0^{i_*\ti r_*}$ and the reaction coefficients
$\wh\kappa_r$ encode the ``dissipative'' properties.

Because we are going to use the energy-dissipation principle, we
explicitly state the cosh-type dissipation functional given by
\begin{equation}
  \label{eq:fD.cosh}
  \fD(c)=\int_{0}^T \!\! \big\{
\cR(c,\dot{c})+\cR^{*}(c,-\D\cE(c))\big\} \dd t \ = \ 
 \int_{0}^T \!\! \big\{
\cR(c,\dot{c})+\calS(c) \big\} \dd t .
\end{equation}
We will mostly write the dissipation functional $\fD$ in
the first ``$\cR{+}\cR^*$ form'' to highlight its duality
structure. However, for mathematical purposes it will be advantageous
to use the second representation via the \emph{slope function} 
\begin{equation}
  \label{eq:SlopeFcn}
  \calS(c): \bfC \to {[0,\infty[};\ c \mapsto
  \calS(c):=\sum_{r=1}^{r_*} 2\wh\kappa_r \delta_r^*\Big( 
  \big(\frac{c^{\alpha^r}}{c_*^{\alpha^r}}\big)^{1/2} -
  \big(\frac{c^{\beta^r}}{c_*^{\beta^r}}\big)^{1/2}  \Big)^2,
\end{equation}
which is continuous on $\bfC$ and satisfies
$\calS(c)=\cR^*(c,{-}\D\cE(c))$ for $c\in \bfC_+$. 
Sometimes $\calS$ is also called the
(discrete) Fisher information as it corresponds to $\int_\Omega 4k |\nabla
\sqrt\rho|^2 \dd x = \int_\Omega k |\nabla\rho|^2/\rho \dd x $ in the diffusion case. 
 
A special feature of DBRS is that all equilibria have the property
that they provide an equilibrium to each individual reaction $r\in R$,
where we do not need linear independence of $(\gamma^r)_{r\in R}$, see
also \cite[Sec.\,2]{MiHaMa15UDER} or \cite{Miel17UEDR}.

\begin{lem}[Equilibria in DBRS]\label{le:Equilibria}
Let $(A,B,c_*,\wh\kappa)$ be a DBRS with slope function $\calS$
defined in \eqref{eq:SlopeFcn}. Then,  the following identities hold: 
\begin{align}
  \label{eq:R=0.S*=0}
\begin{aligned}  \mathscr{E}_{\bm R}& :=\bigset{c\in \bfC}{
    \bm{R}(c)=0} \ = \ \bigset{c\in \bfC}{
    \calS(c)=0} \\
& = \ \bigset{c\in \bfC}{ \forall\,r\in R:\
    \tfrac{c^{\alpha^r}}{c_*^{\alpha^r}} =
    \tfrac{c^{\beta^r}}{c_*^{\beta^r}} } \, . 
\end{aligned}
\end{align}
Moreover, if $\wt c_*\in \mathscr{E}_{\bm R} \cap
\bfC_+$, then the two DBRS $(A,B,c_*,\wh\kappa)$ and $
(A,B,\wt c_*,\wh\kappa)$  generate the same RRE.
\end{lem}
\begin{proof} \emph{Step 1.}
For $c \in \bfC_+$ the gradient structure $\bm
R(c)=\partial\cR^*(c,-\D\cE(c))$ of the DBRS gives 
\begin{equation}
  \label{eq:Duality}
  \cR(c,\bm R(c)) + \cR^*(c,-\D\cE(c)) = - \D\cE(c)\cdot \bm R(c).
\end{equation}
Thus, $\bm R(c)=0$ implies $\calS(c)= \cR^*(c,-\D\cE(c)) =0$, and
since $\calS(c)$ is the sum of $r_*$ nonnegative terms (cf.\
\eqref{eq:SlopeFcn}) we conclude $   \tfrac{c^{\alpha^r}}{c_*^{\alpha^r}} =
    \tfrac{c^{\beta^r}}{c_*^{\beta^r}}$ as desired. 

\emph{Step 2.} If $c\in \pl\bfC$ satisfies $\bm R(c)=0$, then consider $c_\delta=
c{+}\delta c_*\in \bfC_+$ for $\delta\in
{]0,1[}$. With $|\bm R(c_\delta)|\leq C_0\delta $, 
$\cR(c_\delta,v) \geq 0$, and
$|\D\cE(c_\delta)|\leq i_* \log(1/\delta)$ we find 
\[
\calS(c_\delta) = \cR^*(c_\delta,-\D\cE(c_\delta)) = -
\D\cE(c_\delta)\cdot \bm R(c_\delta)-\cR(c_\delta, \bm
R(c_\delta))  \leq i_*  C_0 \delta \log(1/\delta) + 0 \,.
\]
Using the continuity of $\cS$ we obtain $\calS(c)=\lim_{\delta \to
  0^+} \calS(c_\delta)=0$ and conclude as in Step 1. 
 
\emph{Step 3.} The equilibrium condition of Step 1 implies 
$\wt c_*^{\beta^r}/ \wt c_*^{\alpha^r}=c_*^{\beta^r}/ c_*^{\alpha^r}= :\mu_r^2$
for all $r \in R$. Because in the RRE \eqref{eq:RRE.sym} only the terms 
$\delta^r_* /c_*^{\alpha^r} = \big( c_*^{\beta^r}/
c_*^{\alpha^r}\big)^{1/2} = \mu_r$ and $\delta^r_* /c_*^{\beta^r} =
1/\mu_r$ appear, the last statement follows.  
\end{proof}

The next lemma shows that $\cR(c,\cdot)$ forbids velocities $v$ outside of
the stoichiometric subspace $\Gamma$. Moreover, for all trajectories
$c:[0,T]\to  \bfC$  with $\fD(c)<\infty$, which are much more than the
solutions of the RRE \eqref{eq:RRE}, we find that they have to lie in
one stoichiometric subset $\bfC_q$, i.e.\ the conserved quantities are
already encoded in $\fD$.  

Below we use the characteristic function $\chi_A$ of convex analysis,
which is defined via $\chi_A(v)=0$ for $v\in A$ and $\chi_A(v)=\infty$
otherwise. 
 
\begin{lem}[Conserved quantities via $\fD$]
\label{le:ConservQuanti} Let $\Gamma$, $Q$, $\bfC_q$, and
$\mathscr{Q}$ be defined as in Section \ref{su:MassAction}, and let
$\cR^*$ and $\cR$ be defined as in 
\eqref{eq:RRE.GS.b} and \eqref{eq:def.cR}, respectively. 

(a) For all $(c,v)\in \bfC\ti \R^{i_*}$ we have $\calR(c,v) \geq
\chi_\Gamma(v)$. 

(b) If $c\in \rmW^{1,1}([0,T];\bfC)$ satisfies $\fD(c)<\infty$, then 
 $Q\dot{c}=0$ a.e., or equivalently there exists $q\in \mathscr{Q}$ such
 that $c(t)\in \bfC_q$ for all $t\in [0,T]$. 
\end{lem}
\begin{proof} Using $\gamma^r = \alpha^r-\beta^r$ we find 
$\cR^{*}(c,\xi)=0$ for $\xi\perp\Gamma=\mathrm{ker}(Q)$ and conclude 
\[
\cR(c,v)=\sup_{\xi}(\xi\cdot v-\cR^{*}(c,\xi))\geq\sup_{\xi\perp\Gamma} 
             ( \xi\cdot v-\cR^{*}(c,\xi))
 =\sup_{\xi\perp\Gamma}(\xi\cdot v)=\chi_{\Gamma}(v).
\]
This proves part (a). 

The bound $\fD(c)<\infty$ implies that $\int_{0}^{T}\cR(c,\dot{c})\dd
t<\infty$ and hence $\dot{c}\in\Gamma=\mathrm{ker}(Q)$ a.e. This
proves $Q\dot{c}(t)=0$ a.e.\ and by the absolute continuity of
$c$, the function $t\mapsto Qc(t)$ must be constant. Hence part (b) is
established as well. 
\end{proof}

\begin{rem}[Different gradient structures]\label{re:SeveralGS}
We emphasize that the symmetric RRE \eqref{eq:RRE.sym}, which was
obtained from the DBC, indeed has many other gradient structures with
the same relative entropy $\cE$ given in \eqref{eq:RRE.GS.a}. 
Choosing arbitrary smooth and strictly convex functions
$\Phi_r:\R\to {[0,\infty[}$
with $\Phi_r(0)=0$ and $\Phi_r({-}\zeta)=\Phi_r(\zeta)$ we may define 
\[
\cR^*_\Phi(c,\xi) = \sum_{r=1}^{r_*} \wh\kappa_r\delta_r^*
\Lambda_r\Big(\frac{c^{\alpha^r}}{c_*^{\alpha^r}} ,
                \frac{c^{\beta^r}}{c_*^{\beta^r}} \Big) \,
\Phi_r\big(\,(\alpha^{r}{-}\beta^{r})\cdot\xi\,\big) \quad
\text{with }  \Lambda_r( a,b ) = \frac{a\ - \ b}{\Phi'_r\big( \log a{-}\log b\big)}
\]
and $\delta_r^*=\big(c_*^{\alpha^r} c_*^{\beta^r}\big)^{1/2}$. 
Note that $ \Lambda_r$ can be smoothly extended by
$\Lambda_r(a,a)=a/\Phi''_r(0)$. 

To show that the gradient system $(\bfC,\cE,\cR_\Phi)$ indeed
generates \eqref{eq:RRE.sym} as the associated gradient-flow equation,
it suffices to consider the $r^\text{th}$ reaction pair, 
because the dual potential $\cR^*_\Phi$ is
additive in the reaction pairs. Inserting $\D\cE(c)= \big(
\log(c_i/c^*_i)\big)_{i=1,..,i_*}$ 
we obtain the relation 
\begin{align*}
&\D_\xi\cR_{\Phi_r}^*\big(c,{-}\D\cE(c)\big) = \wh\kappa_r \delta_r^* \Lambda_r
\Big(\frac{c^{\alpha^r}}{c_*^{\alpha^r}} , 
      \frac{c^{\beta^r}}{c_*^{\beta^r}} \Big)   \:
\Phi'_r\big(\,(\alpha^{r}{-}\beta^{r})\cdot({-}\D\cE(c))\,\big)
(\alpha^{r}{-}\beta^{r}) 
\\
& \overset{\text{\eqref{eq:Logarithm}}}= -\wh\kappa_r \delta_r^* \Lambda_r
\Big(\frac{c^{\alpha^r}}{c_*^{\alpha^r}} , 
      \frac{c^{\beta^r}}{c_*^{\beta^r}} \Big)   \:
\Phi'_r\Big(\,\log\big(\frac{c^{\alpha^r}}{c_*^{\alpha^r}}\big) - 
   \log\big(\frac{c^{\beta^r}}{c_*^{\beta^r}} \big) \,\Big)
(\alpha^{r}{-}\beta^{r}) 
=  - \wh\kappa_r \delta_r^* \Big(\frac{c^{\alpha^r}}{c_*^{\alpha^r}}  -  
      \frac{c^{\beta^r}}{c_*^{\beta^r}} \Big) \:(\alpha^{r}{-}\beta^{r}) ,
\end{align*} 
which is the desired result. 

The choice $\Phi_r(\zeta)=\zeta^2/2$ was used in \cite{Miel11GSRD},
while here we use $\Phi_r=\mathsf C^*$ leading to 
\[
\Lambda_r(a,b)=(ab)^{1/2} \quad \text{and} \quad 
 \delta_r^* \,\Lambda_r \Big(\frac{c^{\alpha^r}}{c_*^{\alpha^r}} , 
      \frac{c^{\beta^r}}{c_*^{\beta^r}} \Big)  = \big( c^{\alpha^r}
      c^{\beta^r}\big)^{1/2}.
\]
This is the desired term in \eqref{eq:RRE.GS.b} that is independent of
$c_*$, while for other choice of $\Phi_r$ the last term will depend on
$c_*$ (see \cite{MieSte19?CGED}).
\end{rem}

\subsection{Fast-slow reaction-rate equation}
\label{su:FastSlowRRE}

We assume that some reactions are fast with reaction coefficients
$\wh\kappa_r^\eps = \kappa_r/\eps$,  
while the others are slow with reaction coefficients $\wh\kappa_r^\eps
= \kappa_r$ (of order 1). Here we assume that the set or reaction
indices $R=\{1,...,r_*\}$ decomposes into  
$R_\slow\,\dot{\cup}\,R_{\fast}$. For simplicity we assume
that the detailed-balance steady state $c_*$ is independent of $\eps$,
but a soft dependence with a limit $c^\eps_*\to c_*\in \bfC_+$ could
be allowed as well. 
\begin{equation}
  \label{eq:RRE.fs}
  \begin{aligned}
&\dot{c}  = \bm{R}_\eps(c)=-\sum_{r=1}^{r_{*}} \wh\kappa^\eps_r
\delta_r^* \Big(\frac{c^{\alpha^{r}}}{c_*^{\alpha^{r}}} -
\frac{c^{\beta^{r}}}{c_*^{\beta^{r}}} \Big)\big(\alpha^{r}-\beta^{r}\big)
 =\bm{R}_\slow(c)+\frac{1}{\eps}\bm{R}_\fast(c) \\
& \text{with }\bm{R}_{\mathrm{xy}}(c)=\sum_{r\in R_{\mathrm{xy}}} \kappa_r
\delta_r^* \Big(\frac{c^{\alpha^{r}}}{c_*^{\alpha^{r}}}
\frac{c^{\beta^{r}}}{c_*^{\beta^{r}}}
\Big)\big(\alpha^{r}-\beta^{r}\big) \quad \text{for } \ \mathrm{xy}
\in \{ \slow, \fast\}. 
\end{aligned}
\end{equation}

Obviously, for each $\eps>0$ we have a cosh-type gradient
structure $(\bfC,\cE,\calR_\eps)$ with 
\begin{equation}
  \label{eq:cR*.eps}
  \cR^*_\eps(c,\xi) = \cR^*_\slow(c,\xi) + \frac1\eps
  \cR^*_\fast(c,\xi) \quad \text{with }
  \cR^*_{\mathrm{xy}}(c,\xi) = \!\!\sum_{r\in R_{\mathrm{xy}}} \!\!
  \kappa_r \sqrt{c^{\alpha^r}c^{\beta^r}} \: \sfC^*\big(
  (\alpha^r{-}\beta^r)\cdot \xi\big).  
\end{equation}
The aim of this paper is to investigate the behavior of the gradient
structures $(\bfC,\cE,\cR_\eps)$ in the limit $\eps\to0^+$. In
particular, we study the $\Gamma$-limit of the induced
dissipation functionals $\fD_\eps$ obtained as in \eqref{eq:fD.cosh}
but with the duality pair $\cR_\eps{+}\cR^*_\eps$. 
 
At this stage we report on well-known results (see e.g.\
\cite{Both03ILRC, DiLiZi18EGCG}) about the limit evolution for $\eps
\to 0^+$. For small times of order $\eps$ the fast system $\bm{R}_\fast$
will dominate,
while for $t\in [\sqrt\eps, T]$ a slow dynamics takes place where the
  slow reactions drive the evolution and the fast reactions remain in
  equilibrium. 

To be more precise we introduce the fast time scale $\tau=t/\eps$ such
that in terms of $\tau$ we obtain the rescaled system $c'(\tau) = \eps
\bm{R}_\slow(c(\tau)) + \bm{R}_\fast(c(\tau))$. For $\eps \to 0^+$ we
obtain the fast system 
\begin{equation}
  \label{eq:RRE.fast}
  c'(\tau) = \bm{R}_\fast(c(\tau)), \quad c(0)=c_0.
\end{equation}
This is again a RRE satisfying the detailed-balance condition and all
constructions introduced in Sections \ref{su:MassAction} and
\ref{su:DBCiGS}. In particular we obtain  the fast stoichiometric subspace 
\[
\Gamma_\fast:=\mathOP{span}\bigset{ \gamma^{r} \in \Z^{i_*}}{ r\in R_\fast}
\ \subset \ \Gamma \subset \R^{i_*}. 
\]
For the annihilator $ \Gamma_\fast^\perp:= \bigset{q\in \R^{i_*} }{
  \forall\,\gamma\in \Gamma_\fast:\ q\cdot \gamma =0 }$ we have
$\Gamma^\perp \subset \Gamma_\fast^\perp$ and $m_\fast:=\dim
\Gamma_\fast^\perp \geq m= \mathOP{dim}\Gamma^\perp$. Thus, we can
extend the basis $\{q_1,\ldots,q_m\}$ for $\Gamma^\perp$ to a basis
$\{q_1,...,q_m,...,q_{m_\fast}\}$ for $\Gamma_\fast^\perp$ and
define the conservation operator $Q_\fast: \R^{i_*} \to \R^{m_\fast}$ via
\[
Q_\fast^\top := \big( q_1,\ldots,q_{m_\fast}\big):  \R^{m_\fast} \to
\R^{i_*} \quad \text{and set } \Qspace :=
\bigset{ Q_\fast c\in\R^{m_\fast} }{ c\in \bfC }. 
\]
In particular, the important defining relations of $Q_\fast $ are 
\begin{equation}
  \label{eq:Def.Qfast}
  \mathOP{ker} Q_\fast = \Gamma_\fast \quad \text{and} \quad \mathOP{im}
  Q_\fast^\top = \Gamma_\fast^\perp.
\end{equation}
Of course, our interest lies in the case
$0\leq m \lneqq m_\fast \lneqq i_*$. In that case the mapping
$c\mapsto Qc$ yields fewer conserved quantities for the full fast-slow
RRE \eqref{eq:RRE.fs} than the mapping $c\mapsto \sfq = Q_\fast c$ supplies
for the fast RRE \eqref{eq:RRE.fast}. We call $\sfq \in \Qspace$ the slow
variables, as they may still vary on the slow time scale. In particular, the
decomposition of $\bfC$ into fast stoichiometric subsets
\begin{equation}
  \label{eq:Q.q.fast}
   \bfC = {\text{\larger[1]$\cup$}}_{
    \sfq\in \Qspace }  \bfC _{\sfq}^\fast
    \quad\text{where}\quad
    \bfC _{\sfq}^\fast :=\bigset{ c\in \bfC  }{ Q_\fast c=
    \sfq } 
\end{equation}
is finer than $\bfC = \bigcup_{q\in \mathscr{Q}}\bfC_q$. 

\newcommand{\POS}{\!\;\rule{0pt}{0.7em}}

Starting
from a general initial condition $c_0$, one can show that the
solutions $c^\eps:[0,T] \to \bfC$ of the fast-slow RRE
\eqref{eq:RRE.fs} 
have a limit $c^0:[0,T] \to \bfC$, but this limit may not be continuous at
$t=0$. On the short time scale $\tau=t/\eps$ we may define 
$\wt c\POS^\eps(\tau)= c^\eps(\eps\tau)$ which has a limit $\wt c\POS^0 : 
{[0,\infty[} \to \bfC$ satisfying the fast RRE
\eqref{eq:RRE.fast} and having a limit  $\ol c_0:=
\lim_{\tau\to \infty} \wt c\POS^0(\tau)$ with $\bm{R}_\fast(\ol
c_0)=0$.   Hence, we define the set of fast equilibria
\begin{equation}
  \label{eq:def.scrEfast}
  \scrEfa :=\bigset{ c\in \bfC  }{\bm{R}_\fast(c)=0 } 
 =\bigset{ c\in \bfC }{ \forall\, r\in R_\fast:\;
   \tfrac{c^{\alpha^{r}}}{c_*^{\alpha^{r}}}
 = \tfrac{c^{\beta^{r}}}{c_*^{\beta^{r}}} }
\end{equation}
such that for $\tau\in {[0,\infty[}$ the solution $\wt c\POS^0(\tau)$
describes the approach to the slow manifold and
$\ol c_0 \in \scrEfa $. On the time scale of order $1$, the
limits $c^0(t)$ of the solutions $c^\eps(t)$ satisfy
$c^0(t)\in \scrEfa $ for all $t\in {]0,T]}$, and one has
the matching condition $\ol c_0=\lim_{t\to 0^+} c^0(t)$.

The evolution of the solutions $c^0$ within $\scrEfa$ is driven by the slow
reactions only; the fast reactions keep the solution on the fast-equilibrium manifold
$\scrEfa$. In particular, it can be shown (see \cite{Both03ILRC, DiLiZi18EGCG}
or \cite{MieSte19?CGED} for the linear case) that $c^0$ satisfies the limiting
equation
\begin{equation}
  \label{eq:Limit.RRE}
  \dot{c}(t)=\bm{R}_\slow(c(t))+\lambda(t) \quad \text{with
  } c(t) \in \scrEfa \ \text{ and } \ 
 \lambda(t) \in\Gamma_\fast, \quad c(0)=\ol c_0. 
\end{equation}
The result of our paper is quite different: We will pass to the limit
in the gradient systems $(\bfC,\cE,\cR_\eps)$ directly and obtain an
effective gradient system $(\bfC,\cE,\cR_\eff)$, see Theorem
\ref{thm:tiltEDPcvg}. As a consistency check, we will show in
Section \ref{se:EffGS} that the gradient-flow equation for
$(\bfC,\cE,\cR_\eff)$ is indeed identical to the limiting equation 
\eqref{eq:Limit.RRE}, see Proposition \ref{pr:LimEqnConstra}.

\subsection{A simple example for a fast-slow system}
\label{su:Example}

As a guiding example, we consider a reaction system consisting of
three species $X_i$, $i=1,2,3=i_*$, which interact via
$r_*=2$ reactions,  one being slow and one being fast: 
\begin{align*}
\text{slow: } \ X_1  \rightleftharpoons X_3\qquad 
\text{fast: } \ X_{1}+X_{2}  \rightleftharpoons 2X_{3}\,.
\end{align*}
Hence, the stoichiometric vector are given by 
\begin{align*}
\alpha^{1} & =(1,0,0)^\top,\quad \beta^{1}=(0,0,1)^\top, \quad
\gamma^{1}=(1,0,-1)^\top,
\\
\alpha^{2} & =(1,1,0)^\top,\quad \beta^{2}=(0,0,2)^\top, 
 \quad \gamma^{2}=(1,1,-2)^\top\,.
\end{align*}
Hence, one can easily check that the space of conserved quantities
is $\mathrm{span}\left((1,1,1)^\top\right)\in\R^{3}$ which
defines the matrix $Q=(1,1,1)\in\R^{1\times3}$. 

We have $R=R_\fast\cup R_\slow = \{1\} \cup \{2\}$ and the RRE reads 
\begin{align}
\label{eq:RRE.exa}
\dot c &=\bm{R}_{\eps}(c)= 
  (c_3{-}3c_1) \begin{pmatrix}1\\[-0.1em] 0\\[-0.1em]-1\end{pmatrix} 
  + \frac1\eps (c_3^2{-}c_1c_2) 
 \begin{pmatrix}1\\[-0.1em] 1\\[-0.1em] -2\end{pmatrix}.
\end{align}
The nontrivial equilibria of this RRE are given by
$c_*=(c_1^*,c_2^*,c_3^*)^\top=  \sigma (1,9,3)^\top$ for
$\sigma> 0$. All these $c_*$ satisfy the detailed balance condition, and
\eqref{eq:RRE.exa} takes the symmetric form
\eqref{eq:RRE.sym}, viz.\ 
\begin{align*}
&\dot{c}=
   -\kappa_1 \delta_1^*\Big(\frac{c_1}{c_1^*} -\frac{c_3}{c_3^*}\Big) 
    \begin{pmatrix}1\\[-0.1em] 0\\[-0.1em]-1\end{pmatrix} 
   - \frac{\kappa_2}\eps\,\delta_2^* 
    \Big(\frac{c_1c_2}{c_1^*c_2^*} - \frac{c_3^2}{(c_3^*)^2}\Big) 
  \begin{pmatrix}1\\[-0.1em] 1\\[-0.1em] -2\end{pmatrix} 
\\ 
&\text{with }\delta_1^*=(c_1^*c_3^*)^{1/2}=\sigma\sqrt3,\quad
\delta_2^*=(c_1^*c_2^*)^{1/2}c_3 =9\sigma^2, \quad \kappa_1=\sqrt3,
\quad \text{and } \kappa_2=1.
\end{align*}
Thus, we find the cosh-type gradient structure $(\bfC,\cE,\cR^*_\eps)$ of
Section \ref{su:DBCiGS} with
\begin{align*}
&\cE(c)= \sigma \LB (c_1/\sigma) + 9\sigma
 \LB (c_2/(9\sigma)) + 3\sigma  \LB (c_3/(3\sigma)),
\\
&\cR_{\eps}^{*}(c,\xi)=\sqrt{3c_{1}c_{3}}\: \sfC^{*}(\xi_1{-}\xi_3)
+ \frac{1}{\eps}\sqrt{c_{1}c_{2}c_{3}^{2}}\: 
  \sfC^{*}(\xi_1{+}\xi_2{-}2\xi_3)\, .
\end{align*} 
As noted just after Proposition~\ref{pr:GS.cosh},
$\cR^*_\eps$ is independent of $c_*$.

The associated fast system consists simply of one reaction, hence we
find 
\[
\Gamma_\fast = \mathOP{span}(1,1,-2)^\top, \quad 
Q_\fast = \bma{ccc}1&1&1\\ 1&-1&0 \ema,  \quad 
\Qspace = \bigset{\sfq\in \R^2} {\sfq_1\geq 0 }.
\]
The stoichiometric sets $\bfC_q$ with $Qc=q\in \R^1$ are triangles,
which decompose into the straight segments $\bfC^\fast_\sfq$ given by
$Q_\fast c= \sfq$, whereas the set of fast equilibria
\[
\scrEfa =\bigset{
  c\in \bfC  }{ c_{1}c_{2}= c_{3}^{2}}\,.
\]
is curved. See  Figure \ref{fig:FastSlowDecomp} for an illustration. 
\begin{figure} 
\centerline{
\begin{minipage}[b]{0.5\textwidth}
\centerline{
  \includegraphics[width=0.95\textwidth, trim = 0 135 420 200, clip=true]
     {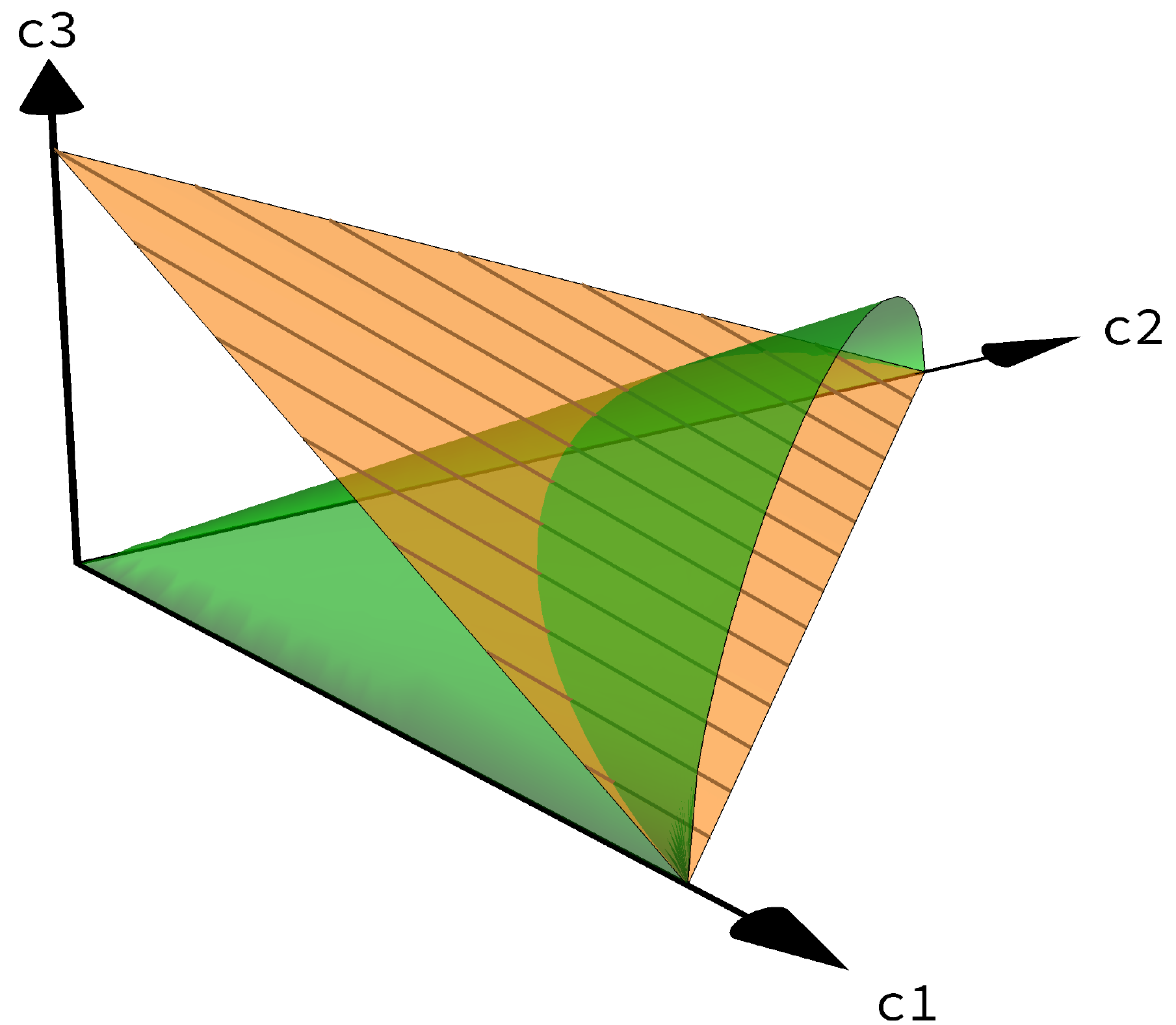}}
\end{minipage}
\begin{minipage}[b]{0.4\textwidth}
\caption{The state space $\bfC={[0,\infty[}^3$ decomposes into the
  triangles $Qc=c_1+c_2+c_3=q$ (light brown), which decompose into the straight
  segments $Q_\fast c= \sfq$ (brown). Each segment has exactly one
  intersection with the fast equilibria $\scrEfa$ (green).}
\end{minipage}}
\label{fig:FastSlowDecomp}
\end{figure}

Finally, we discuss the evolution for our example starting with the
initial condition $c_0=(10,4,0)^\top$ such that
$Qc^\eps(t)=c_1^\eps(t)+c_2^\eps(t)+c_3^\eps(t)=14$ is the conserved
quantity. 
Since there is only one fast reaction, the second conserved
quantity $c_1-c_2=\sfq_2=6$ shows that $\wt c^\eps(\tau)=c^\eps(\eps
t)$ converges to $\wt c(\tau)$ and $\wt c(\tau) \to \ol c_0=(8,2,4)^\top \in
\scrEfa $ for $\tau \to \infty$.

Thus, the limit solution $c^0$ satisfies the limiting equation
\eqref{eq:Limit.RRE}, which reads in our case
\[
\dot c = (c_3{-}3c_1) \begin{pmatrix}1\\[-0.1em]
  0\\[-0.1em]-1\end{pmatrix}  + \lambda_0 \begin{pmatrix}1\\[-0.1em]
  1\\[-0.1em]-2\end{pmatrix}, \quad c_1(t)c_2(t)=c_3(t)^2, \quad
c(0)=\ol c_0=(8,2,4)^\top.
\]  
By eliminating the Lagrange multiplier $\lambda_0\in \R$ and using
the conserved quantity $Qc=14$ this system is equivalent to the system
\[
\dot c_1-\dot c_2 = c_3-3c_1, \quad c_1c_2 = c_3^2, \quad
c_1+c_2+c_3 = 14.
\]

Simulations are shown in Figure \ref{fig:Simul}, which show the fast
convergence to $\scrEfa$ and then the slow convergence  to
the final steady state $c_\mathrm{eq}= (1,9,4)^\top$.
\begin{figure}
\begin{minipage}{0.43\textwidth}
\includegraphics[height=0.5\textwidth]{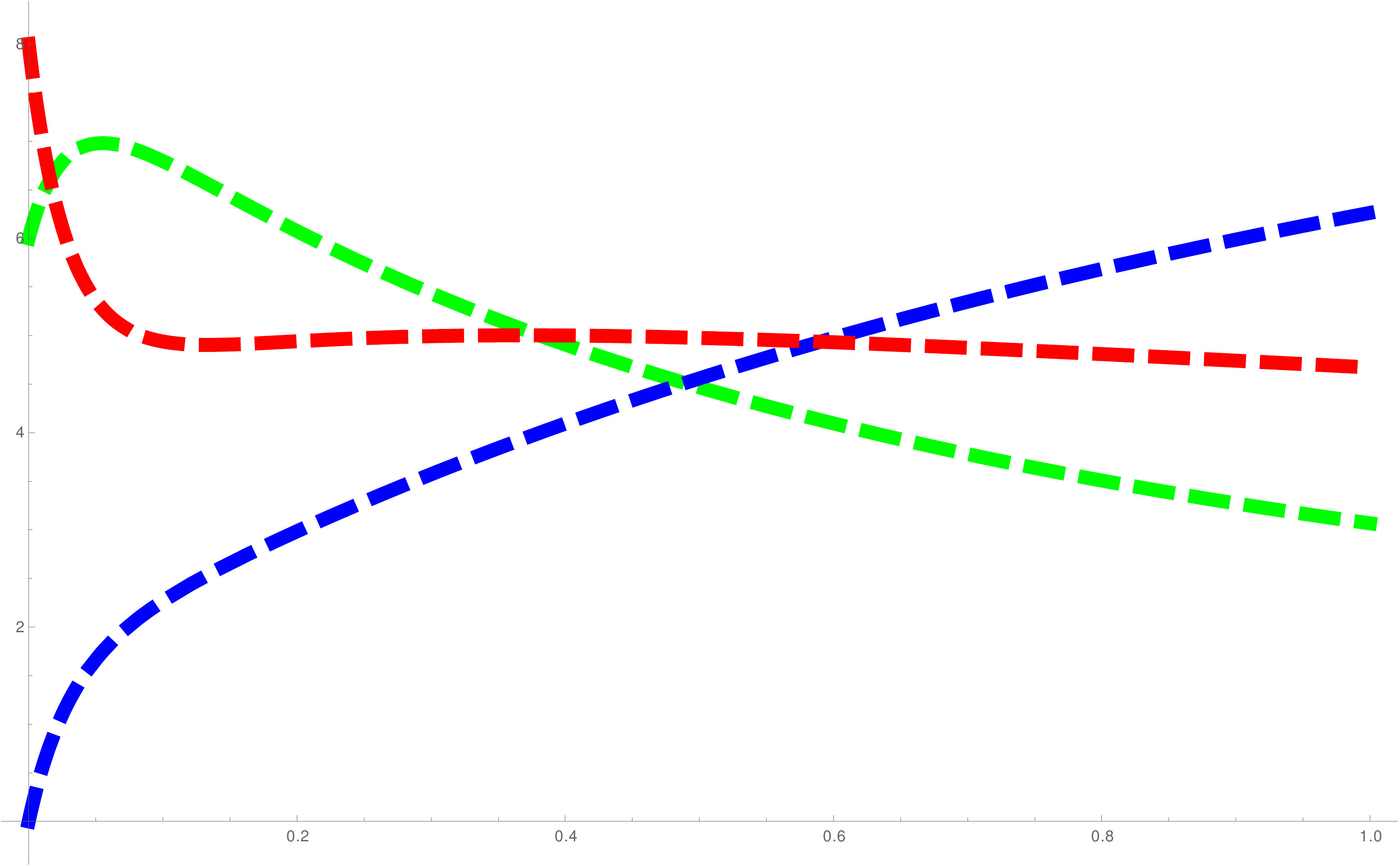}
\!\!\!\!\!\raisebox{5em}{$\ba{c}
c_2(t)\\[0.5em] c_3(t)\\[0.5em] c_1(t)\ea$}
\\
\includegraphics[height=0.5\textwidth]{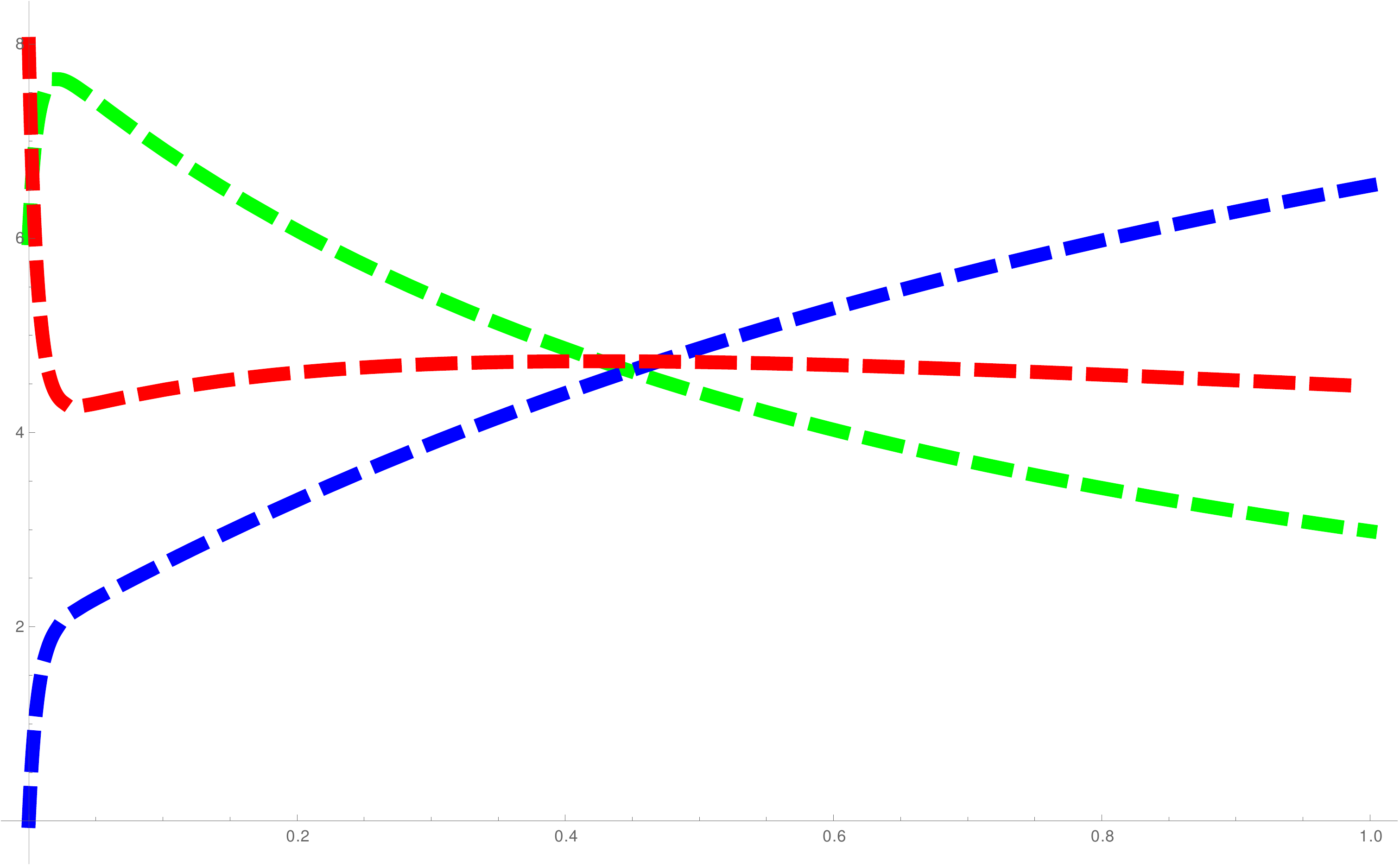}
\!\!\!\!\!\raisebox{5em}{$\ba{c}
c_2(t)\\[0.5em] c_3(t)\\[0.5em] c_1(t)\ea$}
\end{minipage}
\begin{minipage}{0.55\textwidth}
\includegraphics
  [width=\textwidth, trim= 280 20 0 70,
  clip=true]{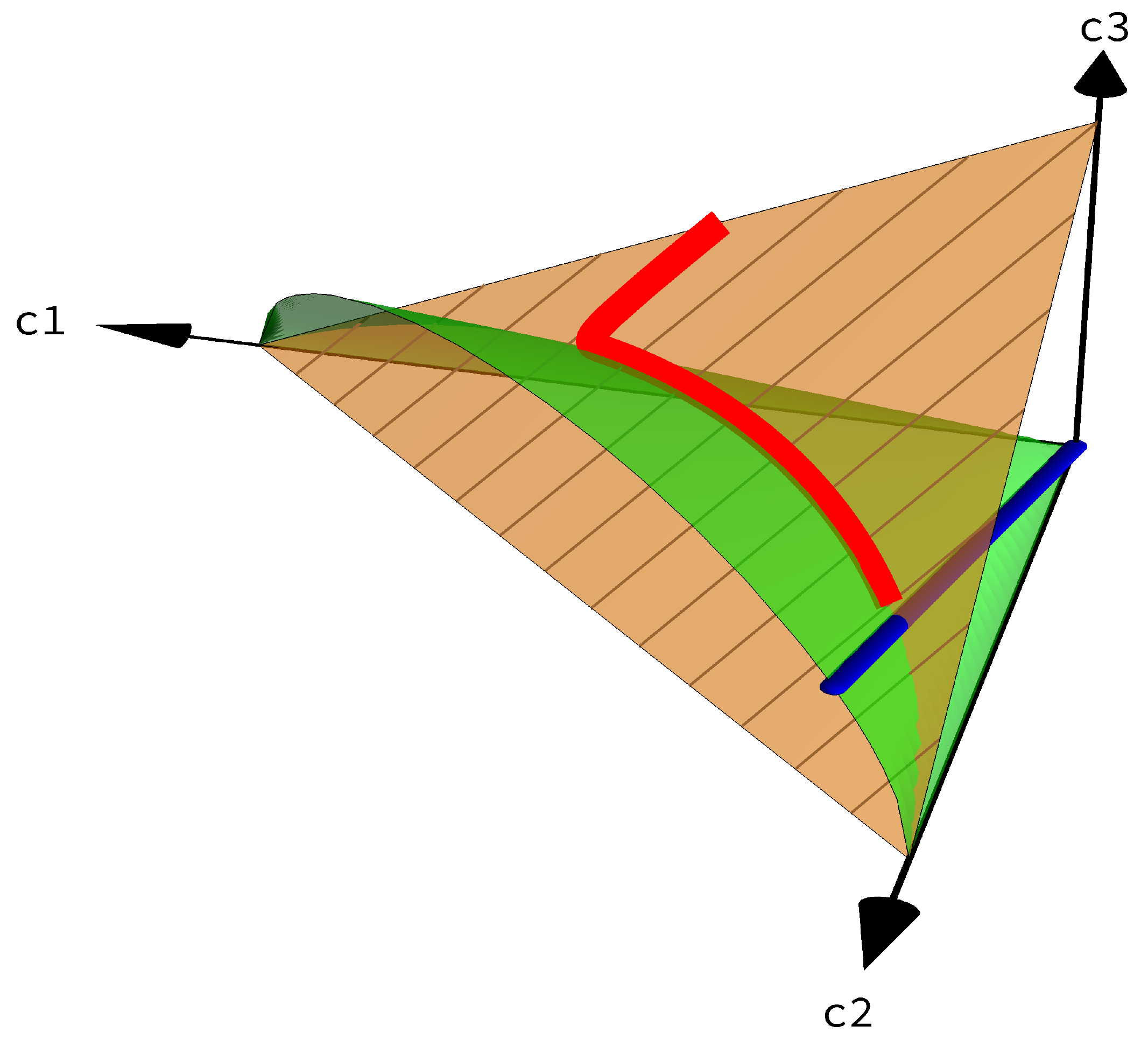}
\end{minipage}
\centerline{\begin{minipage}{0.9\textwidth}
\caption{ Numerical calculation of the solutions $c^\eps(t)$ 
for the RRE \eqref{eq:RRE.exa} with $c^\eps(0)=(4,0,10)^\top$ with
$\eps =1$ (upper left) and
    $\eps=0.2$ (lower left). The lower left figure shows the fast
    convergence to $\ol c_0=(8,2,4)^\top$. The right graphs displays the curve $t
\mapsto c^\eps(t)$ (red), which lies in the plane $Qc=14$ (light brown). It
quickly approaches $\mathscr{M}_\slow$ (green) and then moves towards
the set of steady states (blue).}
\end{minipage}}
\label{fig:Simul}
\end{figure}

\section{EDP-convergence and effective gradient structure}
\label{se:EDPcvg.EffGS}

In this section we first provide the precise definitions of
EDP-convergence for gradient systems. Next we present the our main
result concerning the EDP-limit of the cosh-type gradient structure
for the fast-slow RRE with detailed-balance condition as introduced in
Section \ref{su:FastSlowRRE}, where the proofs are postponed to later
sections. Finally, in Section \ref{se:EffGS} we discuss the
obtained effective gradient system $(\bfC,\cE,\cR_\eff)$ and show that
the induced gradient-flow equation indeed is the same as the limiting
equation \eqref{eq:Limit.RRE}.

\subsection{Definition of different types of EDP-convergence}
\label{su:Def.EDPcvg}

The definition of EDP-convergence for gradient systems relies on
the notion of $\Gamma$-convergence for functionals (cf.\
\cite{Dalm93IGC}).  If $Y$ is a Banach
space and $I_\eps:Y\to \R_\infty$ we write $I_\eps \Gammlim I_0$ for
$\Gamma$-convergence in the strong topology, which is defined via the 
liminf and limsup estimates: 
\begin{align*}
\Gamma\text{-liminf:}\quad& \quad w_\eps \to w_0 \quad \Longrightarrow \quad
\liminf\nolimits_{\eps\to 0^+} I_\eps(w_\eps) \geq I_0(w_0) ,
\\ 
\Gamma\text{-limsup:}\quad& \forall\, \wh w_0 \in Y\ \exists (\wh w_\eps)_\eps:
\quad \big(\  \wh w_\eps \to \wh w_0 \text{ and } 
\limsup\nolimits_{\eps\to 0^+} I_\eps(\wh w_\eps) \leq I_0(\wh w_0)\ \big).
\end{align*}
If in both conditions the strong convergence $\to$ is replaced by weak
convergence $\weak$, then we have (sequential) weak $\Gamma$-convergence and
write $I_\eps \Gammlim I_0$. If weak and strong $\Gamma$-convergence holds,
this is called Mosco convergence and written as $I_\eps \Moscolim I_0$.

For families of gradient systems $(X,\cE_\eps,\cR_\eps)$, three
different levels of EDP-convergence are introduced and discussed in
\cite{DoFrMi19GSWE, MiMoPe18?EFED}, called simple EDP-convergence,
EDP-convergence with tilting, and contact EDP-convergence with
tilting. Here we will only use the first two notions. For all three
notions the choice of weak or strong topology is still to be decided
according to the specific problem. Here in the state space
$X=\R^{i_*}$ this question is irrelevant, but it is relevant for
curves $u:[0,T] \to  \sfX  $ lying in $Y= \rmL^1([0,T];X)$, where
the state space $ \sfX $ is a closed convex subset with non-empty
interior of the Banach space $X$. For our paper, the strong topology
will be sufficient. 

\begin{defn}[Simple EDP-convergence]
  \label{def:EDPcvg} A family of gradient structures
  $( \sfX ,\cE_{\eps},\cR_{\eps})$ is said to \emph{EDP-converge}
  to the gradient system 
  $( \sfX ,\cE_{\mathrm{0}},\cR_{\mathrm{eff}})$ if the
  following conditions hold:
\begin{enumerate}
\itemsep0.1em
\item $\cE_{\eps}\Gammlim \cE_{0}$ on $ \sfX \subset X$;
\item $\fD_{\eps}$ strongly $\Gamma$-converges to $\fD_{0}$
  on $\L^{1}([0,T]; \sfX )$ conditioned to bounded energies 
 (we write $\fD_{\eps}\GamLimE\fD_{0}$), i.e.\ we have\vspace*{-0.5em}%
{\setlength{\leftmarginii}{1.5em}%
\begin{enumerate}
\itemsep-0.01em 
\setlength{\leftmargin}{4.6em} 
\item (Liminf) For all strongly converging families
  $u_{\eps} \to  u$ in $\L^{1}([0,T]; \sfX )$ 
which satisfy  $\sup_{\eps>0}\mathrm{ess\,sup}_{t\in[0,T]}\cE_{\eps}(u_{\eps}(t))<\infty$, 
we have $\liminf_{\eps\to0^+}\fD_{\eps}(u_{\eps})\geq\fD_{0}(u)$.
\item (Limsup) For all $\wt u\in\L^{1}([0,T]; \sfX )$ there exists a 
 strongly converging family
$\wt u_{\eps} \to  \wt u$ in $\L^{1}([0,T]; \sfX )$ with 
 $\sup_{\eps>0}\mathrm{ess\,sup}_{t\in[0,T]}\cE_{\eps}(\wt u_{\eps}(t))<\infty$
and $\limsup_{\eps \to  0^+}\fD_{\eps}( \wt
u_{\eps}) \leq \fD_{0}( \wt u)$;
\end{enumerate}
}
\item there is an effective dissipation potential
  $\cR_{\mathrm{eff}}:  \sfX \times X \to \R_{\infty}$ such that
  $\fD_0$ takes the form of a dual sum, namely
  $\fD_{0}(u)=\int_{0}^{T} \{\cR_{\mathrm{eff}}(u,\dot{u}) {+}
  \cR_{\mathrm{eff}}^{*}(u,-\D\cE_{\mathrm{eff}}(u))\} \dd t$.
\end{enumerate}
\end{defn}

Similarly, one can also use weak $\Gamma$ or Mosco convergence conditioned to
bounded energy, which we will then write as $\fD_{\eps}\wGamLimE\fD_{0}$ and
$\fD_{\eps}\MoscoLimE\fD_{0}$.  In fact, for our fast-slow reaction systems we
are going to prove $\fD_{\eps}\MoscoLimE\fD_{0}$.

A general feature of EDP-convergence is that under suitable conditions the
gradient-flow equation $\dot u = \pl_\xi\calR^*_\eff(u,{-}\D\cE_0(u))$ of the
effective gradient system $(X,\cE_0,\cR_\eff)$ is indeed the limiting equation
equation for the family $\dot u =\pl_\xi\calR^*_\eps(u,{-}\D\cE_\eps(u))$,
i.e.\ limits $u^0$ of solutions $u^\eps$ of latter equations solve the former
equation, see e.g.\ \cite[Thm.\,11.3]{Brai14LMVE},
\cite[Lem.\,3.4]{MieSte19?CGED}, or \cite[Lem.\,2.8]{MiMoPe18?EFED}. For our
case, because of the degeneracy of the fast variables, such a result
requires an assumption that the initial states $u_\eps(0)$ converge to the slow
manifold as $\eps\to0$. Propositions \ref{pr:LimitEqn.q} and
\ref{pr:LimEqnConstra} describe this in detail.

A strengthening of simple EDP-convergence is the so-called
\emph{EDP-convergence with tilting}. This notion involves the tilted
energy functionals $\cE^\eta_\eps:  \sfX  \ni u \mapsto
\cE_\eps(u) - \langle \eta,u\rangle$, where the tilt $\eta$ (also called
forcing) varies through the whole dual space $X^*$. 

\begin{defn}[{EDP-convergence with tilting 
(cf.\ \cite[Def.\,2.14]{MiMoPe18?EFED})}]\label{def:tiltEDP}
A family of gradient structures $( \sfX ,\cE_{\eps},\cR_{\eps})$
is said to \emph{EDP-converge with tilting} to the gradient system 
$( \sfX ,\cE_0,\cR_\eff)$, if for all tilts $\eta\in X^{*}$ we have
 $( \sfX ,\cE^\eta_{\eps},\cR_{\eps})$ EDP-converges to
 $( \sfX ,\cE^\eta_{\eps},\cR_\eff)$. 
\end{defn}

In \cite[Sec.\,2.4]{MiMoPe18?EFED} the admissible tilts are chosen to be
general $\rmC^1$ functions leading to tilted families
$\calE_\eps^\calF=\calE_\eps + \calF$. This choice was inevitable because
there the underlying space $\bfQ$ was a manifold. In the present paper the
underlying space $\bfC$ is a convex subset of a linear space $\sfX$ which allows
for the simpler definition. 

We observe that $\cE_\eps \Gammlim \cE_0$  implies $\cE_\eps^\eta
\Gammlim \cE_0^\eta$ for all $\eta\in X^*$ (and similarly for weak
$\Gamma$-convergence), since the linear tilt $u\mapsto -\langle \eta,
u\rangle $ is weakly continuous. The main and nontrivial assumption is
that additionally  
\[
\fD_\eps^\eta: u \mapsto \int_0^T\!\!\big\{ 
\cR_\eps(u,\dot u) + \cR^*_\eps(u,\eta{-}\D\cE_\eps(u)) \big\} \dd t
\]
$\Gamma$-converges in $\L^1([0,T]; \sfX )$ to $\fD_0^\eta$ for all
$\eta \in X^*$ and that this limit $\fD_0^\eta$ is given
in dual-sum form with $\cR_\eff$ via 
\[
\fD_{0}^{\eta}(u)=\int_{0}^{T}\!\!\big\{ \cR_{\mathrm{eff}}(u,\dot{u})+ 
\cR_{\mathrm{eff}}^{*}(u,\eta{-}\D\cE_{\mathrm{eff}}(u)) \big\} \dd t.
\]
The main point is that $\cR_\eff$ remains independent of $\eta\in
X^*$. We refer to \cite{MiMoPe18?EFED} for a
discussion of this and the other two notions of EDP-convergence.

\subsection{Our main EDP-convergence result}
\label{su:MainEDPcvg}

Since we have assumed that the stationary measure does not
depend on $\eps>0$, also the energy $\cE_{\eps}=\cE$ is
$\eps$-independent. Since $\cE$ is also convex and lower semicontinuous,
we have the trivial Mosco convergence $\cE_{\eps}\Moscolim \cE$.

To study the $\Gamma$-limit of the dissipation functionals $\fD_\eps$
we first extend them to the space 
\[
\L^1([0,T];\bfC): = \bigset { c\in\L^{1}([0,T];\R^{i_{*}}) }{ 
 c(t)\in \bfC  \text{ a.e.}} \, .
\]
For this  we also use the slope functions (where $\mathrm{xy}\in \{\fast,\slow\}$)
\begin{equation}
  \label{eq:Slope.eps}
  \calS_\eps(c)=\calS_\slow(c)+ \frac1\eps \calS_\fast(c) \quad
\text{with }\calS_\mathrm{xy}(c) = \sum_{r\in R_\mathrm{xy}}
2\kappa_r \delta^*_r \Big( 
 \big(\frac{c^{\alpha^r}}{c_*^{\alpha^r}} \big)^{1/2}
 - \big(\frac{c^{\beta^r}}{c_*^{\beta^r}} \big)^{1/2} \Big)^2. 
\end{equation}
For $\eps>0 $ the dissipation functional
$\fD_{\eps}: \L^1([0,T];\bfC) \to [0,\infty]$ is now given
by 
\begin{equation}
  \label{eq:fD.extended}
  \fD_{\eps}(c)=\begin{cases} 
  \int_{0}^{T} \!\!\big\{ \cR_{\eps}(c,\dot{c})+\calS_{\eps}(c)
  \big\} \dd t&\text{for } c \in \rmW^{1,1}([0,T];\bfC),\\
 \infty &\text{otherwise}. 
\end{cases} 
\end{equation}
We recall that the dual dissipation potentials are given by (with
$\gamma^r=\alpha^r- \beta^r$) 
\begin{align*}
\cR_{\eps}^{*}(c,\xi) &
=\cR_\slow^{*}(c,\xi)+\frac{1}{\eps}\cR_\fast^{*}(c,\xi)
 \quad \text{with }\cR^*_\mathrm{xy}(c,\xi)= \sum_{r \in R_\mathrm{xy}}
\kappa_r \sqrt{c^{\alpha^r}c^{\beta^r}}\: 
  \sfC^*\big(\gamma^{r}\cdot\xi\big)\, .
\end{align*}
Because $\calS_\fast(c)\geq 0$ and $\cR^*_\fast(c,\xi)\geq 0$ we
observe that $\calS_\eps(c)$ and $\cR^*_\eps(c,\xi)$ are
monotonously increasing for $\eps\downarrow 0$. 
Thus, their
$\Gamma$-limits exist and are equal to the pointwise limits, which are denoted by $\cS_0$
and $\cR^*_0$ respectively (this uses 
\cite[Rem.\,5.5]{Dalm93IGC} and the continuity of $\calS_\fast$ and
$\cR^*_\fast$.)

Using \eqref{eq:R=0.S*=0} for the fast system we know that for $c\in
\bfC$ the three conditions $\bm{R}_\fast(c)=0$, $\calS_\fast(c)=0$,
and $c \in \scrEfa$ are equivalent. Hence, we conclude 
\[
  \lim_{\eps\to 0^+}\calS_\eps(c)=:\calS_0(c) = \calS_\slow(c) +
  \chi_{\scrEfa} (c) , \quad 
\text{where }\chi_A(b)= \left\{\ba{cl} 0 &\text{for }b\in A,\\[-0.1em] 
\infty& \text{for }b\not\in A. \ea \right.
\] 
Obviously, we have $\lim_{\eps\to 0^+}\cR^*_\eps(c,\xi)=:\cR^*_0(c,\xi)=0$ for $\xi\in \Gamma_\fast^\perp$
and for $c\in \bfC_+$ we obtain $\cR^*_0(c,\xi)= \infty $
for $\xi \not\in  \Gamma_\fast^\perp$. Thus, we 
define the effective dual dissipation potential as 
\begin{align}
\label{eq:EffDualDissPot}
  \cR_{\eff}^{*}(c,\xi) &
  =\cR_\slow^{*}(c,\xi) + \chi_{\Gamma_\fast^{\perp}}(\xi)\,.
\end{align}
Note that $\cR_{\eff}^{*}(c,\xi) \geq \cR^*_0(c,\xi)$ where inequality
may happen on the boundary of $\bfC$, e.g.\ at $c=0$. Nevertheless, we
have the important relation 
\begin{subequations}
\label{eq:Eff.calS*.cR}
\begin{equation}
  \label{eq:Eff.calS*}
  \forall\, c \in  \bfC_+: \ 
\cR^*_\eff(c,{-}\D\cE(c)) = \calS_0(c):=  \calS_\slow(c) +
  \chi_{\scrEfa} (c) \,. 
\end{equation}
The primal effective dissipation potential $\cR_\eff$ is given by the
Legendre--Fenchel transformation: 
\begin{equation}
  \label{eq:Eff.cR}
  \begin{aligned}
  \cR_{\eff}(c,v) & =\sup_{\xi\in\R^{i_{*}}}\left\{
    v\cdot\xi-\cR_{\eff}^{*}(c,\xi)\right\}
  =\sup_{\xi\in\R^{i_{*}}}\left\{
    v\cdot\xi-\cR_\slow^{*}(c,\xi)-\chi_{\Gamma_\fast^{\perp}}(\xi)\right\}
  \\ 
  & =\inf_{v_{1}+v_{2}=v}\left\{
    \cR_\slow(c,v_{1})+\chi_{\Gamma_\fast}(v_{2})\right\}
  =\inf_{v_{2}\in\Gamma_\fast}\left\{
    \cR_\slow(c,v{-}v_{2})\right\}, 
\end{aligned}
\end{equation}
\end{subequations}
where we have used
$(\chi_{\Gamma_\fast})^*= \chi_{\Gamma_\fast^\perp}$ and the classical
theorem on the Legendre-Fenchel transformation turning a sum into an
infimal convolution (see \cite[Prop.\,3.4]{Atto84VCFO}).

To state our main result we now impose a non-trivial structural 
assumption that is crucial for our result and its proof. An  analogous
condition on the uniqueness of equilibria in each stoichiometric subset
$\bfC^\fast_\sfq$  was  used in \cite[Eqn.\,(17)]{Miel17UEDR}. We believe that
the theory of EDP-convergence can be studied without this assumption,
but then one has to refine the results and the solution technique
suitably, see the counterexample in Remark \ref{re:Counterexa.fD}.

\begin{assumption}[Conditions on the fast equilibria $\scrEfa$]
\label{as:UniqueFastEquil}
For all $\sfq\in\Qspace:=\bigset{Q_\fast c}{c\in \bfC}$, there is exactly one equilibrium
of $c'=\bm{R}_\fast(c)$ in the invariant subset $\bfC^\fast_\sfq$ (cf.\
\eqref{eq:Q.q.fast}), i.e.\ 
\begin{equation}
  \label{eq:UFEC}
\text{(UFEC)} \qquad   \forall\, \sfq \in\Qspace : \quad \#\big( 
\bfC^\fast_\sfq \cap\scrEfa \big) \: =\: 1\, ,
\end{equation}
which is called the \emph{unique fast-equilibrium condition}. 
By $\Psi : \Qspace\to \bfC$ we denote the mapping such that
$\{\Psi(\sfq)\} = \bfC^\fast_\sfq \cap\scrEfa $ for all $\sfq\in
\Qspace$.

We further impose the following positivity assumption on $\Psi$: 
\begin{equation}
  \label{eq:InterEquilCond}
\exists\, \ol \sfq\in \Qspace\   \forall\,\theta\in {]0,1]}\ \forall
\,\sfq\in \Qspace \  
 \forall\, i\in I:\quad \Psi(\sfq {+} \theta \ol\sfq)_i>0 \ 
  \text{ and } \    \Psi(\sfq {+} \theta \ol\sfq)_i\geq \Psi(\sfq)_i.
\end{equation} 
\end{assumption}
The positivity and monotonicity assumption \eqref{eq:InterEquilCond} seems
to be only technical and it is only used at one point, namely in Step 1 in the
proof of Theorem \ref{th:Recovery}. We expect that this assumption can be
avoided by a more careful construction of recovery sequence.

In Section \ref{su:DiscAssumpt} we will show that one of possibly several
equilibria in $ \bfC^\fast_\sfq$ is always given as the minimizer of $\cE$ on
$ \bfC^\fast_\sfq$. Thus, the assumption really means that this
``thermodynamic equilibrium'' is the only steady state. Our main
$\Gamma$-convergence result reads as follows.

\begin{thm}[$\Gamma$-convergence]
\label{thm:GammaCvg}Consider a fast-slow DBRS
$(A,B,c_*,\wh\kappa^\eps)$ as in
\eqref{eq:RRE.fs} together with its cosh-type gradient structure
$(\bfC,\cE,\cR_\eps)$ as in Proposition \ref{pr:GS.cosh} and the
dissipation functional $\fD_\eps$ defined in \eqref{eq:fD.extended}.
Moreover, let Assumption \ref{as:UniqueFastEquil} be satisfied.

Then we have  $\fD_{\eps}\MoscoLimE\fD_{0}$
on $\L^{1}([0,T], \bfC )$ conditioned to bounded energies,  where
$\fD_0:\L^1([0,T];\bfC) \to [0,\infty]$ is defined as
\[
\fD_{0}(c):=\begin{cases} 
 \int_{0}^{T} \!\big\{\cR_{\eff}(c,\dot{c})+\cS_0(c) \big\} 
\dd t&\text{for }c\in \rmC^0([0,T];\bfC) \text{ and }Q_\fast c\in
\rmW^{1,1}([0,T];\R^{m_\fast}),\\
 \infty& \text{otherwise}, \end{cases}
\]
where $\cR_{\eff}$ and $\calS_0$ are defined in \eqref{eq:Eff.calS*.cR}.
\end{thm}

The proof of this result is the content of Section
\ref{sec:Proof-of-Theorem}. 

We emphasize that the integrand of $\fD_0$ is (i) degenerate (non-coercive) in
$\dot \sfq$ and (ii) singular (taking the value $\infty$). Concerning (i), we
recall that the definition of $\cR_\eff$ in \eqref{eq:Eff.cR} implies that
$\calR_\eff(c,\cdot)$ vanishes on $\Gamma_\fast$. In fact, it is only possible
to control the time derivative of $t \mapsto Q_\fast c(t) \in \R^{m_\fast}$.
Concerning (ii), we observe that $\cS_0$ equals $+\infty$ outside of $\scrEfa$,
which is a manifold of dimension $m_\fast$, and at each
$c\in \scrEfa\cap \bfC_+$ the subspaces $\rmT_c\scrEfa$ and $\Gamma_\fast$ are
transversal, see Section \ref{se:EffGS}. Assumption
\ref{as:UniqueFastEquil} will be needed to avoid jump-type behavior which can
occur otherwise,  see the counterexample discussed in Remark
\ref{re:Counterexa.fD}.

We now come to our main result on the EDP-convergence with tilting for
the cosh-type gradient systems $(\bfC,\cE,\cR_\eps)$ towards the
effective gradient system $(\bfC,\cE,\cR_\eff)$. 

The theorem enables to establish our main result on EDP-convergence
with tilting. The result is a direct consequence of the
$\Gamma$-convergence stated in Theorem \ref{thm:GammaCvg} and the
general fact for the Boltzmann entropy that tilting is equivalent to
changing the reference measure. In fact, introducing the relative
Boltzmann entropy $\calH(c\mathop{|}w)= \sum_{i=1}^{i_*}
w_i \LB (c_i/w_i)$ we have $\cE(c)=\calH(c\mathop{|}c_*)$ and obtain,
for all $\eta \in \R^{i_*}$, the relation 
\begin{equation}
  \label{eq:TiltRelBoltz}
  \cE^\eta(c) = \cE(c)-\eta{\cdot} c = \calH(c\mathop{|}\bbD^\eta c_*) + E_\eta\quad
\text{with } \bbD^\eta c:=(\ee^{\eta_i}c_i)_{i \in I} \text{ and } 
E_\eta = \sum_{i=1}^{i_*} (1{-}\ee^{\eta_i})c^*_i.
\end{equation}
Thus, we observe that tilting of a DBRS $(A,B,c_*,\wh\kappa^\eps)$
only changes the static property, namely the equilibrium $c_*$ into
$\bbD^\eta c_*$, while the dissipative properties encoded in the 
stoichiometric matrices $A$ and $B$ and the reaction coefficients
$\wh\kappa$ remain unchanged. 

\begin{thm}[EDP-convergence with tilting] \label{thm:tiltEDPcvg} 
Under the assumptions of Theorem
\ref{thm:GammaCvg}, the gradient systems $( \bfC ,\cE,\cR_{\eps})$
EDP-converge with tilting to the gradient system
$( \bfC ,\cE,\cR_{\mathrm{eff}})$.
\end{thm}
\begin{proof} \underline{\emph{Step 1. Simple EDP-convergence:}} Since
  $\cE_\eps=\cE$ is continuous we obviously have $\cE_\eps \Moscolim
  \cE$. Moreover, Theorem \ref{thm:GammaCvg} provides $\fD_\eps
  \MoscoLimE\fD_{0}$. Finally, the relation \eqref{eq:Eff.calS*} shows
  that the integrand of $\fD_0$ has the desired dual structure
  $\cR_\eff(c,\dot c){+}\cR^*_\eff(c,{-}\D\cE(c))$, where we used
  \eqref{eq:Eff.calS*}. Thus, we have
  established the simple  EDP-convergence of $(\bfC,\cE,\cR_\eps)$ to
  the effective gradient system $( \bfC ,\cE,\cR_{\mathrm{eff}})$.

  \underline{\emph{Step 2. EDP-convergence with tilting:}} We use that
  $\cE^\eta=\calH(\cdot\mathop{|}\bbD^\eta c_*)+E_\eta$ is of the same type as
  $\cE = \calH(\cdot\mathop{|}c_*)$ if we ignore the irrelevant constant energy
  shift. Clearly, the new fast-slow RRE \eqref{eq:RRE.fs} has the same
  $A$, $B$, $\kappa_r$, $i_*$, and hence $Q_\fast$; only $c_*$ is replaced
  by $\bbD^\eta c_*$. Thus, all
  structural assumptions are the same, and Theorem \ref{thm:GammaCvg}
  is applicable for all $\eta\in \R^{i_*}$. In particular, the UFEC in 
  \eqref{eq:UFEC} holds for the tilted DBRS by Corollary
  \ref{co:AssumpTilt}. Thus, $(\bfC,\cE^\eta,\cR_\eps)$ EDP-converges
  to $(\bfC,\cE^\eta,\cR_\eff)$ according to Step 1. Since the
  effective dissipation potential $\cR_\eff$ is independent of
  $\eta\in \R^{i_*}$, we have shown EDP-convergence with tilting.
\end{proof}

\subsection{Discussion of the UFEC and definition of $\scrMsl$}
\label{su:DiscAssumpt}

Here we first prove properties of the function $\Psi$ that provides the
fast equilibria (see Assumption \ref{as:UniqueFastEquil}). Secondly, we show
that UFEC is invariant under tilting.

The stoichiometric subsets
$\bfC^\fast_\sfq:=\bigset{c\in \bfC}{ Q_\fast c = \sfq}$ are the intersection
of the affine subspace $\bigset{y\in \R^{i_*}}{Q_\fast y = \sfq}$ of dimension
$m_\fast$ with the simplicial convex cone $\bfC={[0,\infty[}^{i_*}$. Hence,
each $\bfC^\fast_\sfq$ is a closed and convex simplex of dimension
$m(\sfq) \in \{0,1,\ldots,m_\fast\}$. The simplex-boundary $\pl\bfC^\fast_\sfq$
of such a simplex is the union of its boundary simplices of dimension
$m(\sfq) -1$. A two-dimensional $n$-gon has $n$ intervals as boundary, and an
interval has $2$ points as boundary. For the case of a point, which is the only
$0$-dimensional simplex, we say that the boundary is empty. We say that an
equilibrium $c \in \scrEfa$ is a \emph{boundary equilibrium} if $c\in \pl\bfC^\fast_\sfq$.
Otherwise $c\in \scrEfa$ is called an \emph{interior equilibrium}. 
  
The following result provides an alternative construction of the mapping
$\Psi:\Qspace \to \bfC$ that is independent of the UFEC \eqref{eq:UFEC}. We
observe that $\Psi$ is defined for every fast DBRS
$(A^\fast, B^\fast, c_*, \kappa^\fast)$ and that $\Psi$ only depends on
$A^\fast{-}B^\fast$ and $c_*$.  The first part of the next result is also
shown in \cite[Prop.\,2.1]{MiHaMa15UDER} or \cite[Lem.\,2.3]{DiLiZi18EGCG}.

\begin{prop}[Existence and continuity of interior equilibria]
  \label{prop:InterEquil} For a fast DBRS\linebreak[4]
  $(A^\fast\!,B^\fast\!,c_*,\kappa^\fast)$ the energy $\cE$ only depends on $c_*$,
  and $Q_\fast$ only depends on $\Gamma_\fast{=}\mathOP{im}(A^\fast{-}B^\fast)$.
  For each $\sfq \in \Qspace$, denote by $\Psi(\sfq)$ the unique 
  minimizer of $\cE$ on $\bfC^\fast_\sfq$. Then, $\Psi(\sfq)$ is the only
  equilibrium of $\dot c = \bm{R}_\fast(c)$ that lies in the interior
  $\bfC^\fast_\sfq\setminus \pl\bfC^\fast_\sfq$. Moreover, the mapping
  $\Psi: \Qspace \to \bfC$ is continuous, and
  $\Psi:\mathOP{int}\Qspace \to \mathOP{int}\bfC$ is analytic.
\end{prop} 
\begin{proof} \emph{Step 1. Uniqueness and existence  of minimizer:} The
existence of a global minimizer follows from the coercivity of $\cE$ and 
the closedness of $\bfC^\fast_\sfq$. The uniqueness follows from the
convexity of $\bfC^\fast_\sfq$ and the strict convexity of $\cE$. 

\emph{Step 2. Interior property:} If $\bfC^\fast_\sfq$ is a singleton
$\{\wh c\}$, then $\Psi(\sfq)=\wh c$ automatically lies in the
interior. If $c^\pl$ is a point in the boundary and $c^\circ$ a point
in the interior of $\bfC^\fast_\sfq$, then there is at least one
$k\in I$ such that $c^\pl_k=0$ and $c^\circ_k>0$. Since
$c_k \mapsto c^*_k  \LB (c_k/c^*_k)$ has slope $-\infty$ at
$c_k=0$, we conclude that $c^\pl$ cannot be a minimizer of
$\cE: c\mapsto \sum_i c^*_i  \LB (c_i/c^*_i)$. Hence,
$\wh c=\Psi(\sfq)$ lies in the interior of $\bfC_\sfq^\fast$.

\emph{Step 3. Unique equilibrium property:} Since $\cE$ is a strict Liapunov
function for the RRE, the minimizer $\Psi(\sfq)$ has to be an
equilibrium. 

 For the uniqueness, we consider first the case
$\mafo{dim}(\bfC_\sfq^\fast)=m_\fast $,in which case interior points in
$\bfC_\sfq^\fast$ lie in $\bfC_+$. Hence, for any other  equilibrium
$c_\rme$ in the interior of $\bfC^\fast_\sfq$ the derivative
$\D\cE(c_\rme)= \big( \log(c^\rme_i/c^*_i)\big)_i$ is well-defined. Moreover,
Lemma \ref{le:Equilibria} implies 
$c_\rme^{\alpha^r}/c_*^{\alpha^r} = c_\rme^{\beta^r}/c_*^{\beta^r}$ for all
$r \in R_\fast$. These two properties yield $\D\cE(c_\rme)\cdot \gamma^r=0$ for
$r\in R_\fast$. But $\D\cE(c_\rme)\in \Gamma_\fast^\perp$ and
$\bfC_\sfq^\fast \subset c_\rme + \Gamma_\fast$ guarantee that $c_\rme$
minimizes the convex functional $\cE$ on $\bfC^\fast_\sfq$, which yields
$c_\rme=\Psi(\sfq)$.

If $\mafo{dim}(\bfC_\sfq^\fast)= m(\sfq)<m_\fast $ then there exists
$I_0\subset I$ with $m_\fast{-}m(\sfq)$ elements such that
$\bfC_\sfq^\fast\subset \bigset{c\in \bfC}{ c_i=0 \text{ for all } i\in I_0}$
and that for interior points $\wt c\in
\bfC_\sfq^\fast\setminus \pl\bfC_\sfq^\fast$ we have $\wt c_i>0$ for $i\not\in
I_0$. Hence, the above argument can be applied to the reduced system for $\wt
c=(c_i)_{i \in I\setminus I_0} $, i.e.\ the components $c_i=0$, $i \in I_0$ are
simply ignored. 

\emph{Step 4. Continuity of $\Psi$: } Consider a sequence
$\sfq_k\to q_\infty$ and let $c_k=\Psi(\sfq_k)$, then we have to show
that $c_k\to c_\infty$.  We set
$\alpha_k=\cE(c_k)=\min\bigset{\cE(c)}{ c \in \bfC^\fast_{\sfq_k}}$
and choose a subsequence $(k_l)$ such that
$\alpha:= \liminf_{k\to \infty} \alpha_k = \lim_{l\to \infty}
\alpha_{k_l}$. By coercivity of $\cE$ we know that $(c_k)$ is bounded
that there exists a further subsequence (not relabeled) with
$c_{k_l}\to \wt c$ and
$Q\wt c = \lim Qc_{k_l}=\lim \sfq_k=\sfq_\infty$. Hence, we obtain the
estimate
\begin{equation}
  \label{eq:cE.contin}
  \cE(c_\infty)\leq \cE(\wt c) =\lim_{l\to \infty} \cE(c_{k_l}) =
\lim_{l\to \infty} \alpha_{k_l} = \alpha.
\end{equation}
Moreover, our given $c_\infty$ and each $\eps>0$ there exists a $\delta>0$ such that 
$Q\big(B^{\R^{i_*}}_\eps(c_\infty)\cap \bfC\big)$ contains the set
$B_\delta^{\R^{m_\fast}}(\sfq_\infty) \cap \Qspace$. Thus, we find a
sequence $(\wh c_k)_{k\in \N}$ with $\wh c_k \to c_\infty$ and $Q\wh
c_k=\sfq_k\to \sfq_\infty$. Since $\cE$ is continuous we conclude 
\[
\cE(c_\infty) =\lim_{k\to \infty} \cE(\wh c_k) \geq \liminf_{k\to
  \infty} \cE(c_k) = \alpha.
\]
With \eqref{eq:cE.contin} we conclude $\cE(\wt c)=\cE(c_\infty)$,
which implies $c_k\to c_\infty=\Psi(\sfq_\infty)$, as desired.  

\emph{Step 5. Analyticity of $\Psi$:} For $\sfq\in \mathOP{int}\Qspace$ we have
$\Psi(\sfq) \in \bfC_+=\mathOP{int}\bfC$. Hence, $c=\Psi(\sfq)$ can be
characterized by the Lagrange principle for constrained minimizers using the
Lagrange function $L(c,\lambda) = \calE(c) - \mu \cdot (Q_\fast c -\sfq)$
with $\mu \in \R^{m_\fast}$. This characterization leads to the equation $F(c,\mu) = (0,\sfq)$, where 
\[
F(c,\mu):=\big(\rmD \calE(c) - Q_\fast^\top \mu, Q_\fast c\big) .
\]
Obviously, $F:\bfC_+\ti \R^{m_\fast} \to \R^{i_*}\ti \R^{m_\fast}$ is analytic,
and we have $F(\Psi(\sfq),\wt\mu(\sfq))=(0,\sfq)$ for a suitable $\wt\mu$.  If
we can show that $\rmD F(\Psi(\sfq),\mu)$ is invertible for all $\sfq \in
\mathOP{int}\Qspace $, then the implicit
function theorem implies that the mapping $\sfq\to (\Psi(\sfq),\wt\mu(\sfq))$
is analytic as well.

The Jacobian of $F(c,\mu)$ is given by
$\D F(c,\mu) = \binom{\D^2 \cE(c)\ -Q^\top_\fast}{ Q_\fast \qquad 0}$, and we
prove that $\D F(c,\mu)$ is invertible by showing that its kernel is
trivial. Let $(w,\eta)$ be such that $\D F(c,\mu)(w,\eta)^\top=0$. We conclude
that $\D^2\cE(c)w=Q^\top_\fast\eta$ and $Q_\fast w=0$. Since $c$ is positive,
the Hessian $\D^2\cE(c)$ is invertible, and hence, we have
$Q_\fast \D^2\cE(c)^{-1}Q^\top_\fast\eta=0$. Multiplying $\eta$ from the left
and using that $\D^2\cE(c)^{-1}$ is a positive matrix, we have
$Q^\top_\fast\eta=0$. Since $Q^\top_\fast$ is injective, we conclude that
$\eta=0$ which implies that also $w=0$ due to
$\D^2\cE(c)w=Q^\top_\fast \eta=0$.
\end{proof} 

 For later use we observe that by construction we have the relation
\begin{equation}
  \label{eq:Psi.Relation}
  Q_\fast \Psi(\sfq) = \sfq \quad \text{for all } \sfq \in \Qspace.
\end{equation}
A crucial role in our further analysis will be played by the image of
$\Psi$, which we call the slow manifold: 
\begin{equation}
  \label{eq:def.Mslow}
  \scrMsl := \mathOP{im}(\Psi) = \bigset{ \Psi(\sfq)}{ \sfq \in
    \Qspace } \ \subset \ \bfC,
\end{equation}
which is a closed set that is contained in the set of the fast
equilibria $\scrEfa$ defined in \eqref{eq:def.scrEfast}.  The UFEC in
\eqref{eq:UFEC} is made to guarantee that $\scrMsl$ contains all the
fast equilibria:
\begin{equation}
  \label{eq:Mslow=graphPsi}
 \text{(UFEC)}  \ \Longleftrightarrow \  \scrEfa = \scrMsl. 
\end{equation}

It is important to emphasize that $\scrEfa$ can be strictly bigger
than $\scrMsl$, but by Proposition \ref{prop:InterEquil} the 
equilibria in $\scrEfa\setminus \scrMsl$ must 
be so-called boundary equilibria, i.e.\ they lie in
$\pl\bfC^\fast_\sfq \subset \pl\bfC$.  (In the case that
$\bfC^\fast_\sfq\subset \pl \bfC$ the equilibrium $\Psi(\sfq)$ lies
in the boundary of $\bfC$, but is not a boundary equilibrium!)

The equilibria on $\scrMsl$ are stable, since they
are global minimizers of the Liapunov function $\cE$ in their
invariant subset. In contrast, possible boundary equilibria are always
unstable, because starting near the equilibrium but in the interior of
$\bfC^\fast_\sfq$ gives a solution moving towards $\Psi(\sfq)$, see
Figure \ref{fig:NotUEC}.  The UFEC may fail if one has
autocatalytic reactions where the product 
$\alpha^r_i\beta^r_i $ is strictly positive for some $i \in I$; see the
example treated in Remark \ref{re:Exa.noUEC}.

The following simple result provides the characterization of the slow
manifold $\scrMsl$ in terms of the potential force $\D\cE(c)$ and the
annihilator of the fast subspace $\Gamma_\fast$.

\begin{lem}\label{le:DEGamFast} Consider a fast DBRS
  $(A^\fast,B^\fast,c_*,\kappa^\fast)$. Then for $c\in \bfC_+$ we have
\[
\D\cE(c)  \in \Gamma_\fast^\perp =\bigset{\xi\in \R^{i_*}}{ \xi\cdot
  \gamma^r \text{ for }r\in R_\fast} \quad \Longleftrightarrow \quad 
c\in \scrMsl.
\] 
\end{lem} 
\begin{proof} Using $\D\cE(c)= \big( \log(c_i/c^*_i)\big)_{i\in I}$
we find, for all $r\in R_\fast$,  
\[
  0 = \D\cE(c)\cdot \gamma^r = \log\Big( 
  \frac{c^{\alpha^r}}{c_*^{\alpha^r}} \;
  \frac{c_*^{\beta^r}}{c^{\beta^r}} \Big)
\quad \Longleftrightarrow \quad 
\frac{c^{\alpha^r}}{c_*^{\alpha^r}} = \frac{c^{\beta^r}}{c_*^{\beta^r}}
\]
With Proposition \ref{prop:InterEquil} and the definition of $\scrMsl$ in
\eqref{eq:def.Mslow} we obtain the desired result. 
\end{proof}

Finally, we show that the UFEC is invariant under tilting. This is a
nice consequence of the fact that tilting in systems satisfying the
DBC allows us easily to follow the changes in the set $\scrEfa$ of
fast equilibria. 

\begin{cor}[UFEC and tilting]\label{co:AssumpTilt} 
  Consider a fast DBRS $(A^\fast,B^\fast,c_*,\kappa^\fast)$ and general tilt
  vectors $\eta\in \R^{i_*}$. Denote by $\scrEfa^\eta$ and $\scrMsl^\eta$ the
  set of equilibria and the slow manifold, respectively, for the fast DBRS
  $(A^\fast,B^\fast,\bbD^\eta c_*,\kappa^\fast)$. Then, the following holds:

(a) $\scrEfa^\eta = \bbD^\eta\scrEfa^0$ and  $\scrMsl^\eta= \bbD^\eta\scrMsl^0$, 

(b) $(A^\fast,B^\fast,c_*,\kappa^\fast)$ satisfies UFEC if and only if
$(A^\fast,B^\fast,\bbD^\eta c_*,\kappa^\fast)$ does so.
\end{cor}
\begin{proof} By Lemma \ref{le:Equilibria} the equilibria $c \in \scrEfa^0$ are
  given by the condition
  $\tfrac{c^{\alpha^r}}{c_*^{\alpha^r}} = \tfrac{c^{\beta^r}}{c_*^{\beta^r}}$
  for all $r\in R_\fast$.  However, changing $c$ and $c_*$ into $\bbD^\eta c$
  and $\bbD^\eta c_*$, respectively, shows that the condition remains the same.

Moreover, for $c\in \bfC_+$ we have $\D\cE^\eta(\bbD^\eta c) =\D \cE(c) $ by
construction. Since $c \in \bfC_+\cap \scrMsl^0$ is equivalent to
$\D \cE(c) \in \Gamma_\fast^\perp$ we have $\bbD^\eta c \in \scrMsl^\eta$.  By
the continuity of $\Psi=\Psi^0$ and $\Psi^\eta$ (see Proposition \ref{prop:InterEquil})
we conclude $\bbD^\eta \scrMsl^0 \subset \scrMsl^\eta$. As $\bbD^\eta$ is
invertible, we can revert the argument and arrive at
$\bbD^\eta \scrMsl^0 = \scrMsl^\eta$.  Thus, (a) is established.

With (a) we see that  $\scrEfa^0= \scrMsl^0$ is equivalent to 
$ \scrEfa^\eta = \scrMsl^\eta$, and (b) is established as well. 
\end{proof}

\subsection{Examples and problems without the UFEC}
\label{su:noUFEC}

In the following two remarks, we firstly provide a few examples where UFEC does
not hold and secondly show that our main result in Theorem \ref{thm:GammaCvg}
fails without UFEC.

\begin{rem}[Examples without UFEC]\label{re:Exa.noUEC}
  The simplest example of a RRE not satisfying the UFEC condition is the
  autocatalytic reaction $2X \rightleftharpoons X$, leading to the RRE
  $\dot c = \frac1\eps (c-c^2)$, where $\Gamma=\Gamma_\fast =\R$ and
  $m=m_\fast=0$. In particular, the fast stoichiometric subset
  $\bfC = \bfC_{\mathsf 0}^\fast={[0,\infty[}$ contains the the interior
  equilibrium $c_*=1$ and the boundary equilibrium $c=0$.

Next, we consider two different fast systems for two species, the first with
the single autocatalytic  
reaction $X_{1}+X_{2}\leftrightharpoons 2X_{1}$ and the second with 
two non-autocatalytic reactions $2X_1 \leftrightharpoons X_2$ and 
$X_1 \leftrightharpoons 2X_2$. The fast RREs read  
\begin{figure}
\begin{center}
\begin{tikzpicture}
\draw[thick, ->] (-0.5,0)--(3.3,0) node[right]{$c_1$};
\draw[thick, ->] (0,-0.5)--(0,3.3) node[right]{$c_2$};
\draw[ultra thick, color=green] (0,0)--(2.5,2.5) 
  node[above,color=black]{$\scrMsl^{(1)}$};
\draw[ultra thick, color=blue] (0,0)--
  node[pos=0.7,left,color=black]{$\scrEfa^{(1)}$}  (0,3.1) ;

\foreach \i in {1,2,3}
   \draw[ thick, color=red] (0,\i)--(\i,0);
\draw[ thick, color=red, ->   ] (0,3)--(1,2);
 \draw[ thick, color=red, ->  ] (0,2)--(.5,1.5);
\draw[ thick, color=red, ->  ] (3,0)--(2,1);
 \draw[ thick, color=red, -> ] (2,0)--(1.5,.5);
\end{tikzpicture}
\qquad \qquad
\begin{tikzpicture}
\draw[thick, ->] (-0.05,0)--(3.3,0) node[right]{$c_1$};
\draw[thick, ->] (0,-0.5)--(0,3.3) node[right]{$c_2$};
\fill  [color=green] (2,2) circle (0.1)
  node[right,color=black]{\qquad $\scrMsl^{(2)}$};
\fill[color=blue] (0,0) circle (0.1)
  node[left,color=black]{$\scrEfa^{(2)}$} ;
\node at (1.5,1.5) {
\includegraphics[width=3cm,trim = 40 40 40 40, clip=true]
                                  {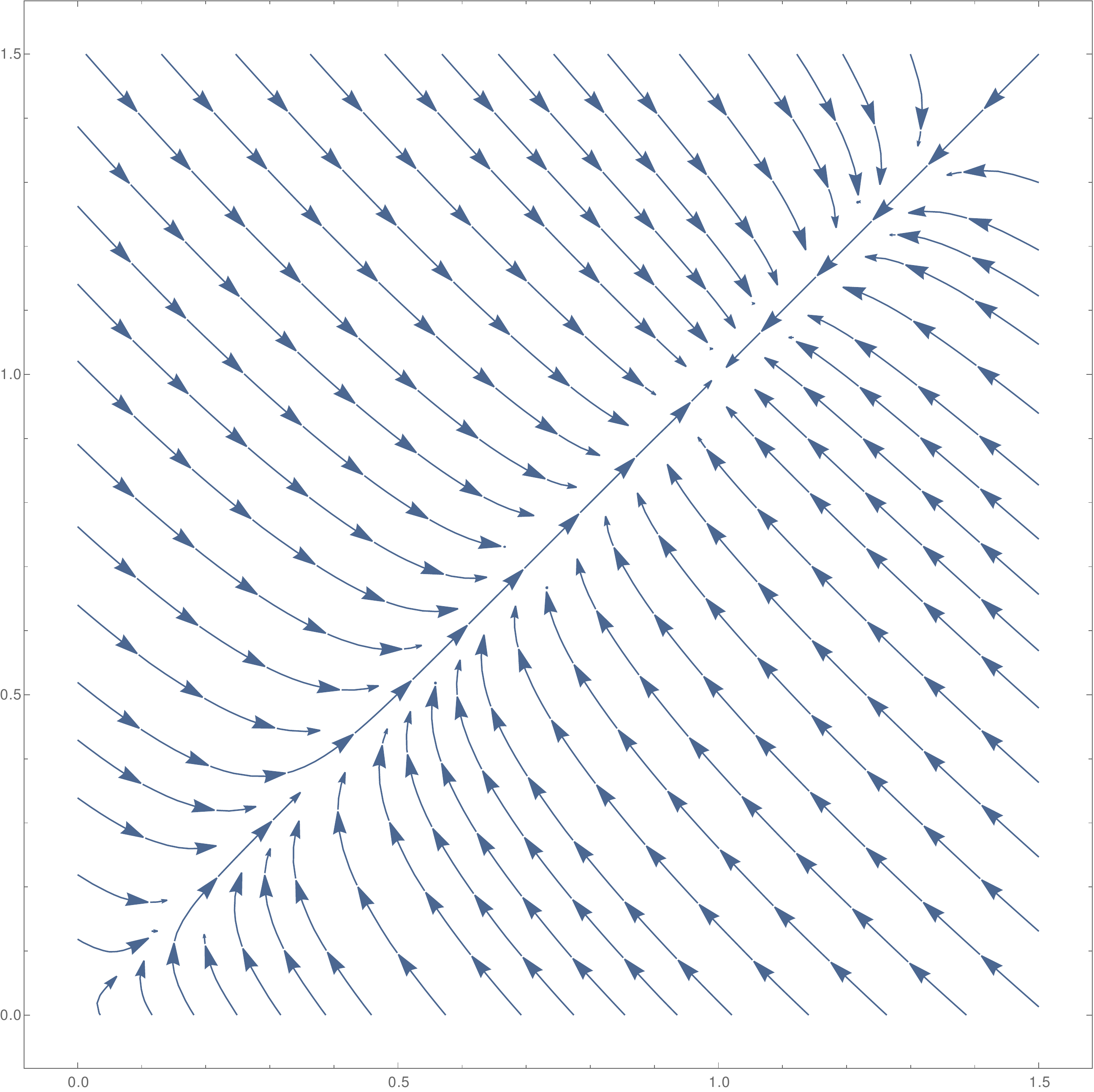}};
\end{tikzpicture}\vspace{-1.5em}
\end{center}
\caption{The slow manifolds  $\scrMsl^{(k)}$ (green) are strictly
  contained in $\scrEfa^{(k)}$ (blue and green). The blue points are
  unstable, while the green points are stable. In the case (1) the invariant
  sets $\bfC^\fast_\sfq$ (red) are one-dimensional, while in case (2)
  we have $\bfC^\fast_0=\bfC$.}
\label{fig:NotUEC} 
\end{figure}
\begin{align*}
c' = \bm R^{(1)}(c)=(c_1^2{-}c_1c_2)\binom{-1}{1}, \qquad 
c' = \bm R^{(2)}(c)=(c_1^2{-}c_2)\binom{-2}{1} + (c_1{-}c^2_2)\binom{-1}{2}. 
\end{align*}
The conserved quantities are given by the matrices 
\[
Q^{(1)}_\fast c = c_1+c_2\in \Qspace^{(1)}={[0,\infty[} \qquad \text{and} \qquad 
Q^{(2)}_\fast c = 0 \in  \Qspace^{(2)}=\{0\}.
\]
The functions $\Psi$ for the minimizers of $\cE$ over
$\bfC^\fast_\bfq$ are given by $\Psi^{(1)}(\sfq) = ( \sfq/2,
\sfq/2)^\top$ and $\Psi^{(2)}(0) = (1,1)^\top$ leading to 
\[
\scrMsl^{(1)} = \bigset{(z,z)^\top}{ z\geq 0} \quad \text{and} \quad
\scrMsl^{(2)} = \{ (1,1)^\top\} \,. 
\]
However, the set of fast equilibria is bigger in both cases:
\[
\scrEfa^{(1)} = \scrMsl^{(1)} \,\dot\cup\, \bigset{(0,z)}{z\geq 0}  
\quad \text{and} \quad \scrEfa^{(2)} = \scrMsl^{(2)} 
  \,\dot\cup\, \{ (0,0)^\top\}.
\]
Figure \ref{fig:NotUEC} displays the invariant sets
$\bfC^\fast_\sfq$, $\scrMsl$, and $\scrEfa$ for both cases.   
\end{rem}

To the knowledge of the authors there are currently no general
sufficient conditions  on the fast DBRS $(A^\fast,B^\fast,
c_*,\kappa^\fast)$ available that guarantee the validity of the 
UFEC. However, in many applications the number $\#(R_\fast)$ of fast
reactions is rather small such that an analysis of the fast RRE is
easily done.

The next remark shows that Theorem \ref{thm:GammaCvg} does
not hold if the UFEC in \eqref{eq:UFEC}  does not hold.

\begin{rem}[A counterexample with jumps]
\label{re:Counterexa.fD} We return to the first RRE $\dot c =
\frac1\eps(c-c^2)$ of the previous remark.  The associated
dissipation potential and slope functions are 
\[
\cR^*_\eps(c,\xi)= \frac{c^{3/2}}\eps\: \mathsf{C^*}(\xi) \quad
\text{and} \quad \calS_\eps(c) = \frac2\eps\:c\,\big(c^{1/2}{-}1\big)^2\,.
\] 
Moreover, the dissipation functional $\fD_\eps$ takes the form
\[
\fD_\eps(c) = \int_0^T\!\Big\{ \frac{c^{3/2}}\eps\, \sfC  
\Big( \frac{\eps \dot c}{c^{3/2}}\Big) + \frac{2c}\eps \big(c^{1/2}{-}1\big)^2
\Big) \Big\} \dd t .
\]
Note that $c\equiv 0$ and $c\equiv 1$ yield $\fD_\eps(c)=0$. Moreover,
fixing $t_*\in {]0,T[}$
the trajectories 
$\wt c{}^\eps(t)= \ee^{(t{-}t_*)/\eps}/(1{+}\ee^{(t{-}t_*)/\eps})$ are exact
solutions of the RRE $\dot c=\frac1\eps (c{-}c^2)$, hence the
energy-dissipation principle gives $\fD_\eps(\wt c{}^\eps)= \cE(\wt 
c{}^\eps(0))- \cE(\wt c{}^\eps(T))\leq \cE(0)-\cE(1)=1$. Thus, the
limit function $\wt c{}^0$ with $\wh c{}^0(t)=0$ for $t<t_*$ and $\wh
c{}^0(t)=1$ for $t>t_*$ is not continuous but must satisfy $\fD_0(\wt c{}^0)\leq 1$,
which is in contradiction to Theorem \ref{thm:GammaCvg}. 

Indeed, using the Modica-Mortola approach as described in
\cite[Sec.\,6]{Brai02GCB} (involving the estimate
$\cR_\eps(c,\dot c)+ \cR^*_\eps(c,{-}\D\cE(c)) \geq -\D\cE(c)\dot c$) it can be
shown that $\fD_\eps \Gammlim \fD_0$ in $\L^1([0,T];\R)$, where $\fD_0$ is
finite only on piecewise constant functions taking values in $\{0,1\}$
only. Moreover, for these functions $\fD_0(c)$ equals the number of jumps times
$\cE(0)-\cE(1)=1$. The same was also observed in \cite{Step19EGCGS}.
\end{rem}

\section{The effective GS and the limiting equation}
\label{se:EffGS}

Here we present two different ways to derive the the limiting
equation from our effective gradient system. The first one is in line
with the coarse-graining approach developed in \cite{MieSte19?CGED},
where a lower-dimensional system is derived for the coarse-grained
variable $\sfq = Q_\fast c$ and the restriction $c=\Psi(\sfq)$ is
built into the model. The second one follows \cite{Both03ILRC} and
\cite[Thm.\,4.5]{DiLiZi18EGCG}, where the variable $c$ is maintained and the
constraint $c \in \scrMsl$ is realized by a suitable projection.


In both cases we start from the $\Gamma$-limit $\fD_0$ of the dissipation
functionals $\fD_\eps$. Combining the $\Gamma$-convergence of $\fD_\eps$
and an assumption of well-preparedness of the initial data, we can take the
limit in the energy-dissipation principle to find
\begin{equation}
  \label{eq:limitingEDB}
  \cE(c(T))+\fD_0(c;0,T)\leq\cE(c(0)) \quad \text{with } \ \fD_0(c;0,T)= 
\int_{0}^{T}\!\! \big\{ \cR_\eff(c,\dot{c})+\calS_0(c)\big\} \dd t.
\end{equation}
From this inequality we  recover the limiting evolution by Theorem
\ref{th:EDP}. 

Any solution satisfies the condition 
\[
\int_0^T \!\! \calS_0(c(t)) \dd t \leq \fD_0(c;0,T) \leq \cE(c(0))- \cE(c(T)) < \infty,
\] 
where by the UFEC the function $\calS_0$ assumes the value $+\infty$
for $c \not\in \scrMsl$. Hence, the continuity of $c$ implies that
$c(t) \in \scrMsl $ for all $t\in [0,T]$. Thus, setting
$\sfq(t):=Q_\fast c(t)$ and using the 
relation \eqref{eq:Psi.Relation}  we have $c(t) = \Psi(\sfq(t))$ for
all $t\in [0,T]$. We recall that the properties 
\[
c \in \rmC^0([0,T];\bfC) \qquad \text{and} \qquad 
\sfq= Q_\fast c \in \rmW^{1,1}([0,T];\R^{m_\fast})
\]
are consequences of Theorem \ref{thm:GammaCvg}.

\subsection{Coarse-graining approach}
\label{su:CoarseGrain} 

In this part we concentrate solely on the slow variables $\sfq$ and define 
\begin{equation}
  \label{eq:ReducedGS}
  \sfE(\sfq):= \cE(\Psi(\sfq)) \quad \text{and} \quad \sfR^*(\sfq, \zeta)
:=\cR^*_\slow(\Psi(\sfq) , Q_\fast^\top \zeta),
\end{equation}
which defines a reduced gradient system $(\Qspace,\sfE,\sfR)$ for the
coarse-grained state  $\sfq \in \Qspace \subset \R^{m_\fast}$. In
particular, $\sfR^*:\Qspace\ti \R^{m_\fast}\to [0,\infty]$ is a  well-defined dual
dissipation potential as $Q_\fast^\top : \R^{m_\fast} \to \R^{i_*}$. 

The main result of this subsection will be that the gradient-flow
equation for the reduced gradient system $(\Qspace,\sfE,\sfR)$ is indeed
the limiting equation and it has a simple representation in terms of
$\bm R_\slow$, $Q_\fast$, and $\Psi$:
\begin{equation}
  \label{eq:reducedRRE}
  \dot \sfq\  = \ \pl_\zeta \sfR^*(\sfq, -\D\sfE(\sfq)) \ = \ 
Q_\fast \bm R_\slow(\Psi(q)) . 
\end{equation}
Thus, $(\Qspace,\sfE,\sfR)$ provides an
exact nonlinear coarse-graining in the sense of
\cite[Sec.\,6.1]{MaaMie20?MCRS}, where the relation $I_{m_\fast} =
Q_\fast \D \Psi(\sfq) $ simplifies the formula for $\sfR^*$ compared
to \cite[Eq.\,(6.2)]{MaaMie20?MCRS}.  
\begin{rem}
This theory is a nonlinear generalization of the coarse-graining
theory developed in \cite{MieSte19?CGED}, where
$\dot{\wh c}= MA_\slow N \wh c$ is the coarse-grained equation. In our
case the role of the reconstruction operator $N:\R^J\to \R^I$ is
played by the nonlinear mapping $\Psi: \Qspace \to \bfC$, while the role
of the coarse-graining operator $M:\R^I\to \R^J$ is our linear
operator $Q_\fast:\bfC \to \Qspace$. 
\end{rem}

The following result provides first the justification of the second identity
in \eqref{eq:reducedRRE}, and then shows that this equation is indeed
the limiting equation obtained from the energy-dissipation principle
for $\cE$ and $\fD_0$. 

\begin{prop}[Reduced gradient structure]\label{pr:ReducedGS}
  Let the DBRS $(A,B,c_*,\wh\kappa^\eps)$ be given as in Section
  \ref{su:FastSlowRRE} and satisfy the UFEC \eqref{eq:UFEC}, and let
  $(\Qspace,\sfE,\sfR^*)$ be defined as above.  Then the following
  identities are valid:\medskip
  
(a) For $\sfq \in \mathOP{int}\Qspace$ we have
\begin{align}\label{eq:ProjectionAndIdentity}
  Q_\fast\D\Psi(\sfq)=I_{m_\fast},\mathrm{~~and,~}\quad
  Q^\top_\fast\D\Psi(\sfq)^\top \text{ is a projection onto } 
  \mathOP{im}(Q^\top_\fast)=\Gamma^\perp_\fast
\end{align}

(b) For $\sfq \in \mathOP{int}\Qspace$ we have $ \pl_\zeta \sfR^*(\sfq,
-\D\sfE(\sfq))  = Q_\fast \bm R_\slow(\Psi(q)) $;\medskip

(c) the primal dissipation potential $\sfR$ takes the form
\[
\sfR(\sfq, w)= \inf\bigset{\cR_\slow(\Psi(\sfq), v) }{ Q_\fast v= w} =
\cR_\eff(\Psi(\sfq), \wt v)  \text{ whenever }
Q_\fast \wt v=w\,;
\]

(d) For $\sfq \in \mathOP{int}\Qspace $ we have $
\sfR^*(\sfq,-\D\sfE(\sfq)) = \calS_\slow(\Psi(\sfq))=: 
\mathsf S (\sfq)$\,;\medskip

(e) $\fD_0(\Psi(\sfq)) = \int_0^T\!\big\{ \sfR(\sfq,\dot \sfq) {+}
\mathsf S(\sfq) \big\} \dd t \,$.
\end{prop}
In part (e) it is crucial to observe that for differentiable
$t\mapsto \sfq(t)$ we cannot guarantee that $t\mapsto c(t)=\Psi(\sfq(t))$ is
differentiable as well, since $\sfq(t)$ need not remain in the interior of $\Qspace$. However, for $\fD_0$ we only need continuity of $c$ and
the differentiability of $t\mapsto Q_\fast c(t)=\sfq(t)$, where we used
$\sfq=Q_\fast \Psi(\sfq)$, see \eqref{eq:Psi.Relation}.
\begin{proof} 
For part (a), we use that $\Psi$ is differentiable in 
$\mathOP{int}\Qspace$. Differentiating the relation $Q_\fast\Psi(\sfq)=\sfq$  yields
$Q_\fast \D \Psi(\sfq)= I_{m_\fast}$. In particular, this implies that $\D\Psi(\sfq)Q_\fast$ is a projection, and hence also its transpose $Q^\top_\fast\D\Psi(\sfq)^\top$.

To show part (b) we first use the chain rule
$\D \sfE(\sfq) = \D\Psi(\sfq)^\top \D\cE(\Psi(\sfq))$. With Lemma
\ref{le:DEGamFast} and part (a) we have that
$\D\cE(\Psi(\sfq))=Q_\fast^\top \D\Psi(\sfq)^\top \D\cE(\Psi(\sfq))$, which
yields
\begin{align*}
\pl_\zeta \sfR^*(\sfq, {-}\D\sfE(\sfq))&= Q_\fast^\top \D_\xi
\cR^*_\slow\big(\Psi(\sfq), {-}Q_\fast^\top \D\Psi(\sfq)^\top
\D\cE(\Psi(\sfq))\big)\\
 &=  Q_\fast^\top \D_\xi
\cR^*_\slow\big(\Psi(\sfq), {-}\D\cE(\Psi(\sfq))\big)
\  = \ Q_\fast \bm
R_\slow(\Psi(\sfq)).   
\end{align*}

For (c) we establish the relation $\sfR^*=\calL \sfR$ via the
Legendre transformation $\calL$:
\begin{align*}
 \big(\calL \sfR(q,\cdot)\big)(\zeta) 
&= \sup \bigset{\zeta\cdot w - \sfR(\sfq, w)}{ w \in \R^{m_\fast}}
\\
&=  \sup \Bigset{\zeta\cdot w + \sup\bigset{-\cR_\slow 
 (\Psi(\sfq), v)}{Q_\fast v=w} }{ w \in \R^{m_\fast}} \\
&= \sup \bigset{\zeta\cdot w -\cR_\slow(\Psi(\sfq), v) }{ Q_\fast v=w} \\
&=   
\sup \bigset{\zeta\cdot Q_\fast v -\cR_\slow(\Psi(\sfq), v) }{ v\in \R^{i_*}} \ = \ 
\cR^*_\slow(\Psi(\sfq), Q^\top_\fast \zeta)\ = \ \sfR^*(\sfq,\zeta).  
\end{align*}

Part (d) follows similarly as part (b) by inserting
$\D \sfE(\sfq) = \D\Psi(\sfq)^\top \D\cE(\Psi(\sfq))$ and
$\D\cE(\Psi(\sfq)) = Q_\fast^\top \D\Psi(\sfq)^\top
\D\cE(\Psi(\sfq))$ into the definition of $\sfR^*$
via $\cR^*_\slow$.  

For part (e) we first observe that $\cS_\slow(\Psi(\sfq))=\mathsf S(\sfq)$ for
all $q\in \Qspace$ by definition. For the rate part $\cR_\eff(c,\dot c)$ part
(c) established that the dependence on $\dot c$ is only through
$Q_\fast \dot c$. But relation \eqref{eq:Psi.Relation} gives
$\frac\rmd{\rmd t} Q_\fast \Psi(\sfq(t))= \dot\sfq(t)$, and the relation
$\cR_\eff(c,\dot c) = \cR_\eff (\Psi(\sfq), \D\Psi(\sfq)\dot\sfq)=
\sfR(\sfq,\dot\sfq)$ holds even $q(t)$ touching the boundary of $\Qspace$.
\end{proof}

The next result shows that the reduced gradient-flow equation
\eqref{eq:reducedRRE} indeed is the limiting equation for the
fast-slow RRE \eqref{eq:RRE.fs} in the sense that for solutions
$c^\eps:[0,T]\to \bfC$ any accumulation point $\sfq:[0,T]\to \Qspace$ of
the family $\big(Q_\fast c^\eps\big)$ solves indeed
\eqref{eq:reducedRRE}. The assumptions on the initial conditions
$c^\eps(0)$ are special to avoid a potential jump at $t=0$, see
Section \ref{su:FastSlowRRE}.  The proof is based on the
energy-dissipation principle and follows
\cite[Thm.\,3.3.3]{Miel16EGCG} or \cite[Lem.\,2.8]{MiMoPe18?EFED} with
some special care because of the degeneracies and singularities of the
limiting problem.

\begin{prop}[Reduced limiting equation]\label{pr:LimitEqn.q} 
  Consider a fast-slow DBRS $(A,B,c_*,\wh\kappa^\eps)$ satisfying the
  UFEC \eqref{eq:UFEC} and let $c^\eps:[0,T]\to \R^{i_*}$ be a family
  of solutions of the fast-slow RRE \eqref{eq:RRE.fs}. If along a
  subsequence (not relabeled) we have $ c^\eps\to c^0$ in
  $\rmL^1([0,T];\bfC)$ and $ c^\eps(0)\to \ol c_0 \in \scrMsl $, then
  $ Q_\fast c^\eps \to \sfq:=Q_\fast c^0 $ weakly in
  $\rmW^{1,1}([0,T];\Qspace)$ and strongly in $\rmC^0([0,T];\Qspace)$, and
  $\sfq$ solves the reduced gradient-flow equation
  \eqref{eq:reducedRRE} with initial condition
  $\sfq(0)=Q_\fast \ol c_0$.
\end{prop}
\begin{proof} The solutions $c^\eps$ satisfy the EDB
$\cE(c^\eps(T))+ \fD_\eps(c^\eps) = \cE(c^\eps(0))$. Using
$c^\eps \to c^0$ in $\L^1([0,T];\R^{i_*})$ and
$\limsup_{\eps\to 0^+}\fD_\eps(c^\eps)\leq \lim_{\eps\to 0^+}
\cE(c^\eps(0))= \cE(\ol c_0)< \infty$, we obtain
$\sfq^\eps:=Q_\fast c^\eps \to \sfq$ weakly in
$\rmW^{1,1}([0,T];\Qspace)$ and strongly in $\rmC^0([0,T];\Qspace)$ by
invoking Theorem \ref{th:Compactness}(ii). Moreover, because of
$\ol c_0\in \scrMsl$ and
$\sfq^\eps(0)=Q_\fast c^\eps(0) \to Q_\fast \ol c_0$ we have
$\sfq(0)=Q_\fast \ol c_0 $ and hence $\ol c_0=\Psi(\sfq(0)) $ and
$\cE(\ol c_0)=\sfE(\sfq(0))$.  Passing to the limit $\eps\to 0^+$
using the liminf estimate in $\fD_\eps \GamLimE \fD_0$ we arrive
at
\[
\sfE(\sfq(T)) + \fD_0(\Psi(\sfq)) \leq \cE(c^0(T)) + 
\fD_0(c^0) \leq \cE(\ol c_0)  = \sfE(\sfq(0)) \,.
\]
Because $ \fD_0(\Psi(\cdot))$ has the
$\sfR{\oplus}\sfR^*$ structure  (cf.\
Proposition \ref{pr:ReducedGS}(d+e)) the energy-dissipation principle
shows that $\sfq$ solves the reduced RRE \eqref{eq:reducedRRE}.
\end{proof}

\subsection{The projection approach} 
\label{suu:Projection}

By contrast to Section \ref{su:CoarseGrain} above, in this section we maintain
the variable $c$.  First, we justify the limiting equation \eqref{eq:Limit.RRE}
with the constraint $c\in \scrMsl$ and the Lagrange multiplier
$\lambda(t) \in \Gamma_\fast$. Secondly, we show that for positive solutions the
evolution can be written as an ODE involving a suitable projection.  Finally,
we compare this to the reduced limiting equation
\eqref{eq:reducedRRE}.\smallskip

\begin{prop}[Limiting equation with constraint]\label{pr:LimEqnConstra}   
  For a fast-slow DBRS $(A,B,c_*,\wh\kappa^\eps)$ satisfying the
  UFEC \eqref{eq:UFEC} we consider a family $c^\eps:[0,T]\to \R^{i_*}$ 
  of solutions of the fast-slow RRE \eqref{eq:RRE.fs}. If along a
  subsequence (not relabeled) we have $ c^\eps\to c^0$ in
  $\rmL^1([0,T];\bfC)$ and $ c^\eps(0)\to \ol c_0 \in \scrMsl $, then
  there exists $c \in \rmC^0([0,T];\bfC)$ such that $c(t)=c^0(t)$
  a.e.\ in $[0,T]$, \ $c(0)=\ol c_0$, \  
  $ Q_\fast c \in \rmW^{1,1}([0,T];\Qspace)$, and $c$ solves the
  limiting equation with constraint:
  \begin{equation}
    \label{eq:LimEqnConstra}
    \dot c(t) = \bm R_\slow(c(t)) + \lambda(t) ,\quad \lambda(t) \in
    \Gamma_\fast, \quad c(t) \in \scrMsl.
  \end{equation}
\end{prop}
\begin{proof} We proceed as in the proof of Proposition
  \ref{pr:LimitEqn.q} but stay with $c$ rather than reducing to $\sfq
  = Q_\fast c$. The solutions $c^\eps$ satisfy the EDB
$\cE(c^\eps(T))+ \fD_\eps(c^\eps) = \cE(c^\eps(0))$. Using
$c^\eps \to c^0$ in $\L^1([0,T];\R^{i_*})$ and
$\limsup_{\eps\to 0^+}\fD_\eps(c^\eps)\leq \lim_{\eps\to 0^+}
\cE(c^\eps(0))= \cE(\ol c_0)< \infty$, we have $Q_\fast c^\eps \to
\sfq$ weakly in $\rmW^{1,1}([0,T];\bfC)$ and strongly in
$\rmC^0([0,T];\R^{i_*})$, see Theorem \ref{th:Compactness}(ii).
With this we define $c(t)=\Psi(\sfq(t))$ for $t\in [0,T]$ such that $c
\in \rmC^0([0,T];\bfC)$ and $Q_\fast c(t)=\sfq(t)$. 

Passing to the limit $\eps\to 0^+$ in the EDB we obtain $\cE(c(T)) +
\fD_0(c) \leq \cE(c(0))$, and the energy-dissipation principle gives
the gradient-flow equation  
\begin{align}
\label{eq:eff.Eqn.c}
\dot{c} &  \in \pl_\xi \cR^*_\eff(c,{-}\D\cE(c)) = \partial_{\xi}\bigg(\cR_\slow^{*}(c,-\D\cE(c))+\chi_{\Gamma_\fast^{\perp}}(-\D\cE(c))\bigg).
\end{align}
For a linear subspace $Y\subset \R^{i_*}$ the set-valued convex
subdifferential $\pl
\chi_{Y^\perp}(\xi) $  equals $Y$ for $\xi \in Y^\perp$ and
$\emptyset$ otherwise, hence the last relation has the form 
\begin{align*}
\dot{c} &
\in\partial_{\xi}\cR_\slow^{*}(c,-\D\cE(c))+\Gamma_\fast= \bm
R_\slow(c)+\Gamma_\fast \quad
\text{and}\quad  \D\cE(c) \in \Gamma^\perp_\fast\ . 
\end{align*}
With Lemma \ref{le:DEGamFast} we can replace the last
constraint by $c\in \scrMsl$, and \eqref{eq:LimEqnConstra} is
established. 
\end{proof}

To obtain an ODE of the form $\dot c = V(c)$ instead of the limiting
equation \eqref{eq:LimEqnConstra} with constraint, we have to resolve
the constraint $\D\cE(c)\in \Gamma_\fast^\perp$. For any curve $s \to
\wt c(s)\in
\scrMsl\cap \bfC_+$ we have $\D\cE(\wt c(s)) \in \Gamma_\fast^\perp$ and
taking the derivative with respect to $s$, we find 
\[
\dot{\wt c}(s) \in \rmT_{\wt c(s)} \scrMsl \quad \text{and} \quad 
\D^{2}\cE(\wt c(s))\dot{\wt c}(s) \in \Gamma_\fast^\perp.
\]
Hence, for $c\in \scrMsl\cap \bfC_+$ the tangent space $\rmT_c \scrMsl
$ of $\scrMsl$ at $c$ is given by 
\[
\rmT_c \scrMsl = \big(\bbH(c)\big)^{-1} \Gamma_\fast^\perp  \quad
\text{with} \quad 
\bbH(c) :=\D^{2}\cE(c) = \mathrm{diag}(1/c_{1},\dots,1/c_{i_*}).
\]

With this we obtain the following
representation of the limiting equation, which matches that in
\cite[Thm.\,2(b)]{Both03ILRC} and \cite[Thm.\,4.5]{DiLiZi18EGCG}. Our
result is more general, since we do not need to assume that the
stoichiometric vectors $\bigset{\gamma^r}{ r\in R_\fast}$ are linearly
independent.

\begin{prop}[Limiting equation for $c\in \bfC_+$]\label{pr:limitEqn.c}
A curve $c:[0,T] \to \bfC_+$ is a solution \eqref{eq:LimEqnConstra} if
and only if 
\begin{equation}
  \label{eq:LimEqnProject}
  \dot c = \big( I-\bbP(c)\big)\, \bm R_\slow(c)  \quad \text{and}
  \quad c(0) \in \scrMsl.
\end{equation}
where the projector $\bbP(c) \in \R^{i_*\ti i_*}$ is defined
via $\mathOP{im}\bbP(c) = \Gamma_\fast$ and
$\mathOP{ker} \bbP(c) =\bbH(c)^{-1} \Gamma_\fast^\perp$. 
\end{prop} 
\begin{proof} \underline{Step 1. \emph{Definition of the projector
      $\bbP(c)$}:} The projector is uniquely defined if $Y_\rmR:=\Gamma_\fast$
  and $Y_\rmK:=\bbH(c)^{-1} \Gamma_\fast^\perp$ provide a direct
  decomposition of $\R^{i_*}$. Assuming $v \in Y_\rmR\cap Y_\rmK$ we
  have $v\in \Gamma_\fast$ and $\bbH(c) v \in
  \Gamma_\fast^\perp$. This implies $v\cdot \bbH(c) v=0$, but since
  $\bbH(c)$ is positive definite we arrive at $v=0$. Hence, $
  Y_\rmR\cap Y_\rmK = \{0\}$. Obviously, $\mathOP{dim} Y_\rmR
  +\mathOP{dim} Y_\rmK= i_*$, so that $\R^{i_*} =  Y_\rmR\oplus Y_\rmK$ is
  established.

\underline{Step 2. \eqref{eq:LimEqnProject} $\Longrightarrow$
    \eqref{eq:LimEqnConstra}:} We set $\lambda(t) = -\bbP(c(t)) \bm
  R_\slow(c) $, and
  with \eqref{eq:LimEqnProject} we obtain 
\[
 \dot c(t) = \bm R_\slow(c(t)) - \bbP(c(t)) \bm R_\slow(c(t))=  \bm
 R_\slow(c(t)) + \lambda(t) \quad \text{ with } \lambda(t)\in \Gamma_\fast.
\]
Moreover, $\bbP(c(t))\dot c(t)= \bbP(c)(I{-}\bbP(c))\bm
R_\slow(c)=0$, which implies $\dot c \in  \bbH(c)^{-1}
\Gamma_\fast^\perp = \rmT_{c(t)}\scrMsl$. Hence, with $ c(0) \in
\scrMsl $  we obtain $c(t) \in \scrMsl $ for all $t\in [0,T]$, and
\eqref{eq:LimEqnConstra} is established. 

\underline{Step 3. \eqref{eq:LimEqnConstra} $\Longrightarrow$
  \eqref{eq:LimEqnProject}:} From $c(t)\in \scrMsl$ we obtain
$\dot c(t)\in \rmT_{c(t)}\scrMsl= \bbH(c(t))^{-1}\Gamma_\fast^\perp $
and conclude $0= \bbP(c) \dot c = \bbP(c) \bm R_\slow(c) +
\bbP(c)\lambda$. Using $\lambda \in\mathOP{im}\bbP(c) =
\Gamma_\fast$ we have $\bbP(c)\lambda= \lambda$ and find 
\[
\big(I- \bbP(c)\big)\, \bm R_\slow(c) = \bm R_\slow(c) - \bbP(c) \bm
R_\slow(c) = \bm R_\slow(c) + \bbP(c)\lambda =  \bm R_\slow(c) +
\lambda = \dot c,
\]
which is the desired equation \eqref{eq:LimEqnProject}.
\end{proof}

To compare the last result with the reduced limiting equation
\eqref{eq:reducedRRE}, we simply use the relation $c(t) =
\Psi(\sfq(t))$ and the fact that $\Psi$ is smooth on
$\mathOP{int}\Qspace $. From this we obtain 
\[
\big(I-\bbP(c)\big)\, \bm R_\slow(c) \ = \ \dot c  \ = \ \rmD\Psi(\sfq)
\dot \sfq\ = \  \rmD\Psi(\sfq(t))Q_\fast \bm R_\slow(\Psi(\sfq)) \ = \
\rmD\Psi(\sfq(t))Q_\fast\bm R_\slow(c).
\]
Thus, we can conclude that for $c=\Psi(\sfq)\in \scrMsl$ we have the
identity 
\[
\big(I-\bbP(c)\big)\ = \ \rmD\Psi(\sfq)Q_\fast,
\]
since the above identity must hold for all possible right-hand sides
$\bm R_\slow$. This can also be shown by using the identity
$c= \Psi(Q_\fast c)$ for all $c\in \scrMsl$ and taking derivatives in the
direction $v\in \Gamma_\fast$ and $w\in \rmT_c\scrMsl$, respectively.  In
particular, this provides the explicit form of the projection of Proposition
\ref{pr:ReducedGS}(a).

\subsection{An example for the effective gradient system}
\label{su:Exa.EffGS} 

In the following example we consider a system with $i_*=5$ species and
$r_*=2$ bimolecular reactions, one fast and one slow. As a result we obtain a
limiting equation with one reaction that is no longer of mass-action type
but involves all species. Taking a further EDP limit (done only
formally)  we recover a trimolecular 
reaction of mass-action type again. 

We consider the following two reactions
\[
\text{fast:}\quad X_{1}+X_{2}\leftrightharpoons X_{3} \qquad
\text{and}\qquad  \text{slow:}\quad X_{3}+X_{4}\rightleftharpoons
X_{5}, 
\]
which give rise to the two stoichiometric vectors 
\[
\gamma^\fast=(1,1,-1,0,0)^\top \quad \text{and} \quad  \gamma^\slow=(0,0,1,1,-1)^\top.
\]
Assuming the steady state $c_*=(1,1,\varrho,1,1)^\top$ and the reaction
coefficients $\wh \kappa^\eps = (\kappa^\fast/\eps,\kappa^\slow)$ the
RRE \eqref{eq:RRE.fs} takes the form
\[
  \dot{c}=-\frac{\kappa^\fast \varrho^{1/2}}{\eps}\big(c_{1}c_{2}-c_{3}/\varrho\big)\gamma^\fast
  - \kappa^\slow \varrho^{1/2} \big(c_{3}c_{4}/\varrho -c_{5}\big) \gamma^\slow \,.
\]
The slow manifold is $\scrMsl=\bigset{c\in {[0,\infty[}^5}{
  c_{1}c_{2}=c_{3}/\varrho} $ and $\Gamma_\fast = \mathOP{span}\gamma^\fast$. With 
\[
Q_\fast=\begin{pmatrix}1 & 0 & 1 & 0 & 0\\
0 & 1 & 1 & 0 & 0\\
0 & 0 & 0 & 1 & 0\\
0 & 0 & 0 & 0 & 1
\end{pmatrix}
\]
we obtain $\Qspace = \mathOP{im} Q_\fast={[0,\infty[}^4$. 
For $\sfq \in \Qspace$ it is easy to compute $\Psi_\varrho(\sfq)$ as a minimizer of
$c\mapsto \cE(c)$ under the constraint $Q_\fast c=\sfq=(q_1,...,q_4)$. We obtain 
\begin{align*}
\Psi_\varrho( \sfq) &= \big( q_1 {-}a_\varrho(q_1,q_2),\  
 q_2 {-} a_\varrho(q_1,q_2), \ a_\varrho(q_1,q_2), \ 
q_3,\  q_4)^\top \in \bfC = {[0,\infty[}^5\\
\text{with }& a_\varrho(q_1,q_2) =\frac1{2\varrho}\Big(
1 + \varrho q_1+ \varrho q_2 - \sqrt{(1{+}\varrho q_1{+}\varrho q_2)^2 -
  4\varrho^2q_1q_2} \Big) \in \big[0,\min\{q_1,q_2\}\big] .
\end{align*}  
In particular, the UFEC \eqref{eq:UFEC} holds.  Moreover, the positivity and
monotonicity condition \eqref{eq:InterEquilCond} can be checked easily with
$\ol\sfq=(1,1,1,1)^\top$. We see that
$c(\theta):=\Psi_\varrho(\sfq{+}\theta\ol\sfq)$ for $\theta \in {]0,1]}$ is
given by
\begin{align*}
c(\theta) &=
\big(q_1{+}\theta -a_\varrho(q_1{+}\theta ,q_2{+}\theta),\: 
 q_2{+}\theta-a_\varrho(q_1{+}\theta, q_2{+}\theta), \: 
 a_\varrho(q_1{+}\theta, q_2{+}\theta), \: 
 q_3{+}\theta,\: q_4{+}\theta \big)^\top.
\end{align*}
Clearly we have $c(\theta)_i>0$, since $c(\theta)_i=0$ would imply $q_i+\theta=0$. 
Differentiating with respect to $\theta$, we obtain 
\begin{align*}
c'(\theta) = \big(1-a_\varrho'[\theta],1-a_\varrho'[\theta], 
 a_\varrho'[\theta], 1,1\big)^\top \quad \text{with } 
a_\varrho'[\theta] = \frac{\varrho(c_1(\theta){+}c_2(\theta))}{1+ \varrho(c_1(\theta){+}c_2(\theta))} , 
\end{align*}
which implies that $c'(\theta)_i>0$. Hence, $\Psi_\varrho(\sfq{+}\theta\ol
q)_i=c(\theta)_i\geq c(0)_i= \Psi(\sfq)$, i.e. the monotonicity condition \eqref{eq:InterEquilCond}
holds. 

We investigate the reduced system. First, we observe that the reduced limiting equation
\eqref{eq:reducedRRE} is given by 
\begin{equation}
  \label{eq:Exa.ReducRRE}
    \dot \sfq = Q_\fast \bm R_\slow(\Psi_\varrho(\sfq)) = -\kappa^\slow
  \varrho^{1/2} \Big( \frac{a_\varrho(q_1,q_2)q_3}{\varrho} - q_4\Big)\,
  \wh\bfgamma \quad \text{with }\wh\bfgamma :=
  Q_\fast \gamma^\slow = (1,1,1,-1)^\top.
\end{equation}
Since $a_\varrho$ is not a monomial, this RRE is no longer of
mass-action type. 

According to Section \ref{su:CoarseGrain} the gradient structure
$(\Qspace,\sfE_\varrho,\sfR_\varrho)$ for \eqref{eq:Exa.ReducRRE} is given via
\begin{align*}
\sfE_\varrho(\sfq)\ \ &= \cE(\Psi_\varrho(\sfq))= \LB(q_1{-}a) + \LB(q_2{-}a)+
\varrho \LB(a/\varrho) + \LB(q_3)+\LB(q_4)\Big|_{a=a_\varrho(q_1,q_2)} ,
\\
\sfR^*_\varrho(\sfq,\zeta)&= \cR_\slow(\Psi_\varrho(\sfq),Q_\fast^\top \zeta) =
\kappa^\slow\, \big(a_\varrho(q_1,q_2)\,q_3q_4 \big)^{1/2}
\:\sfC^*\big(\zeta_1{+}\zeta_2{+}\zeta_3{-}\zeta_4 \big).  
\end{align*}
The energy $\sfE_\varrho$ is no longer of Boltzmann type, because the
previously uncoupled densities $c_1$, $c_2$, and $c_3$ are now
constrained to lie on $\scrMsl$, i.e.\ $c_1c_2=c_3$. Nevertheless, the
form is close to a mass-action type for the trimolecular reaction 
$ Y_1+Y_2+Y_3 \rightleftharpoons Y_4$. 

To recover an exact trimolecular reaction of mass-action type,
one has to perform another limit, namely $\varrho\to 0^+$, which means
that the species $X_3$ is no longer observed, but still exists on a
microscopic reaction pathway. For the limit $\varrho\to 0^+$ we simply
observe the expansion
\[
  a_\varrho(q_1,q_2) = \varrho q_1q_2 + O(\varrho^2) \quad \text{for }
  \varrho\to 0^+,
\]
which implies $\Psi_\varrho(q)\to\Psi_0(q):=(q_1,q_2,0,q_3,q_4)^T$.
If we additionally choose $\kappa^\slow = \ol\kappa /\varrho^{1/2}$ and
insert the expansion for $a_\varrho$ we obtain 
\begin{align*}
\sfE_\varrho(\sfq)\ \  &\to  \ \sfE_0(\sfq) \ = \sum_{j=1}^4 \LB(q_j), \\
\sfR^*_\varrho(\sfq,\zeta) & \to \sfR_0^*( q,\zeta) =\ol\kappa
\,\big(q_1q_2q_3q_4)^{1/2}\, \sfC^*\big(
\zeta_1{+}\zeta_2{+}\zeta_3{-}\zeta_4 \big). 
\end{align*}
Clearly, this is the gradient system generating the RRE of the
trimolecular reaction $ X_1+X_2+X_4 \rightleftharpoons X_5$. Of
course, it is possible to show that this convergence is again a
EDP-convergence with tilting of the gradient systems
$(\Qspace,\sfE_\varrho,\sfR_\varrho)$ to the effective system $(\Qspace,\sfE_0,\sfR_0)$.

\section{Proof of Theorem \ref{thm:GammaCvg}}
\label{sec:Proof-of-Theorem}

Here we will show the $\Gamma$-convergence of the dissipation
functionals, namely $\fD_\eps \MoscoLimE \fD_0$. As usual the proof
consists in three parts: (i) compactness of the sequences $(c^\eps)$
satisfying $\fD_\eps(c^\eps)\leq C$, (ii) the liminf estimate, and the
(iii) the limsup estimate, which needs the construction of recovery
sequences. 

All the following results are derived under the assumptions of
Theorem \ref{thm:GammaCvg}: The fast-slow DBRS
$(A,B,c_*,\wh\kappa^\eps)$ satisfies the unique fast-equilibrium
condition UFEC \eqref{eq:UFEC}. For constructing the recovery sequence in
Section \ref{su:Recovery}, we need additionally the positivity and monotonicity
assumption \eqref{eq:InterEquilCond} for $\Psi$.

\subsection{Compactness}
\label{su:Compactness}

In the definition of $\fD_\eps \MoscoLimE \fD_0$ we consider sequences
$c^\eps \to c^0$ in $\L^1([0,T];\bfC)$ that additionally satisfy
$\sup_{\eps\in {]0,1[},\ t\in [0,T]} \cE(c^\eps(t)) \leq C$.  The
aim is to extract a strongly converging subsequence
$c^{\eps} \to c^{0}$, such that we can talk about pointwise
convergence almost everywhere.  This will be necessary in the liminf
estimate because we cannot rely on convexity, in contrast to the linear
theory developed in \cite{MieSte19?CGED}.  The compactness is
derived via two quite different arguments that complement each other
and reflect the underlying fast-slow structure, which is seen on the
local level via the decomposition of $ \rmT_c \bfC = \R^{i_*}$
in the direct sum of $\Gamma_\fast$ and $\rmT_c \scrMsl$, see Step 1
in the proof of Proposition \ref{pr:limitEqn.c}. First, we derive
time regularity for the slow part of the reactions. Secondly, we prove
convergence towards the slow manifold which then provides the
remaining information for the whole sequence.

\begin{thm}[Compactness via dissipation bound] \label{th:Compactness}
  Consider a family $(c^{\eps})_{\eps>0}$ with
  $c^\eps \rightharpoonup c^0$ in $\L^1([0,T]; \bfC) $, \ 
  $\sup_{\eps>0, \ t\in [0,T]}\cE(c^{\eps}(t))\leq M_\mafo{ener}<\infty$, and
  $\fD_{\eps}(c^{\eps})\leq M_\mafo{diss} <\infty$. Then, we have
\\
(i) \ \ $c^\eps(\cdot)$ is bounded in $\L^\infty([0,T];\bfC)$;
\\
(ii) \  $Q_\fast c^\eps \to Q_\fast c^0$ weakly in
$\W^{1,1}([0,T];\R^{m_\fast})$ and strongly in $\rmC^0([0,T];\R^{m_\fast})$;
\\
(iii) $c^0(t)= \wt c(t):= \Psi(Q_\fast c^0(t)) \in \scrMsl$ for a.a.\ $t\in
[0,T]$, and, in particular $\wt c\in\rmC^0([0,T],\bfC)$;
\\
(iv) \;$c^\eps \to c^0$ in $\L^p([0,T];\bfC)$ strongly for all $p\in
{[1,\infty[}$.
\end{thm} 

\noindent
We emphasize that $c^0$ and $\wt c$ may be different, and this happens
even for solutions, if near $t=0$ a jump develops such that (cf.\
Section \ref{su:FastSlowRRE})
\[
\lim c^\eps (0) =:c_0 \ \neq \ \ol c_0 := \lim_{\tau \to 0^+}
\Big(\lim_{\eps\to 0^+} c^\eps(\tau) \Big).
\] 

Before giving the detailed proof we provide two preliminary results
that underpin the two complementary arguments of the proof. 

For deriving bounds on the time derivatives, one heuristically sees
that for fixed $(c,\xi)$ 
we have $\cR_{\eps}^{*}(c,\xi)\nearrow\cR_{\eff}^{*}(c,\xi)$ as $\eps \to 0$.
By duality, this implies $\cR_{\eps}(c,v)\searrow\cR_{\eff}(c,v)$.
This already shows that control of time derivatives has to be obtained
from $\cR_{\eff}(c,\cdot)$, which only controls $Q_\fast \dot
c$ because $\cR_{\eff}(c,v)=\cR_\eff(c,w)$ if $Q_\fast v= Q_\fast w$,
see \eqref{eq:Eff.cR}.

\begin{prop}[Effective dissipation potential]\label{pr:EffDissPot}
For all $\eps>0$ we have $\cR_\eps(c,v) \geq \cR_\eff(c,v)$ for all
$(c,v)\in \bfC\ti \R^{i_*}$. Moreover, $\cR_\eff$ takes the form
\begin{align*}
\cR_\eff(c,v) &= \wt\cR(c,Q_\fast v) \ \text{ where }
\wt\cR(c,\sfq):= \sup\bigset{\zeta\cdot \sfq -
  \cR^*_\slow(c,Q^\top_\fast \zeta) }{ \zeta \in \R^{m_\fast}}. 
\end{align*}
\end{prop}
\begin{proof} We first use the standard relation from linear
  algebra:
$ 
\mathOP{im}(Q_\fast^\top) = \big(\mathOP{ker}(Q_\fast)\big)^\perp =
\Gamma_\fast^\perp.
$ 
By construction of $\Gamma_\fast$ we have  $\cR^*_\fast(c,\xi)=0$ for
$\xi \in \Gamma_\fast^\perp$ and obtain 
\[
\cR_\eps^*(c,\xi) = \cR^*_\slow(c,\xi) + \frac1\eps \cR^*_\fast(c,\xi)
\leq \cR^*_\eff(c,\xi):=\cR^*_\slow(c,\xi) + \chi_{\Gamma_\fast^\top}(\xi) 
= \cR^*_\slow(c,\xi) + \chi_{\mathOP{im}Q_\fast^\top}(\xi) .
\]
Applying the Legendre-Fenchel transformation we obtain 
\begin{align*}
\cR_{\eps}(c,v) & \geq \cR_\eff(c,v) = 
 \sup \bigset{ v\cdot\xi-\cR_\slow^{*} (c,\xi) }
 {\xi\in\mathOP{im}(Q_\fast^\top)} \\
 & =\sup\bigset{ v\cdot Q_\fast^\top \zeta
   -\cR_\slow^*(c,Q_\fast^\top \zeta) } {\zeta \in \R^{m_\fast} }  =
 \wt\cR(c,Q_\fast v),
\end{align*}
which provides the desired estimate as well as the representation via
$\wt\cR$. 
\end{proof}

The second result concerns the convergence of points towards the slow
manifold $\scrMsl$,  and the crucial property here is
the UFEC \eqref{eq:UFEC} that guarantees the
relation 
\[
 \bigset{\Psi(\sfq)}{ \sfq \in \Qspace \subset \R^{m_\fast}}=:\scrMsl\ \
 \overset{!!}= \ \ \scrEfa \overset{\text{Lemma \ref{le:Equilibria}}}=  
\bigset{ c\in \bfC }{ \calS_\fast(c)=0 }. 
\]

\begin{lem}[Convergence towards $\scrMsl$]\label{le:Cvg2Mslow}
For bounded sequences $(c^n)_{n\in \N}$ in $\bfC$ we have
\begin{equation}
  \label{eq:GenerCont2Msl}
 Q_\fast c^n  \to  \sfq \text{ and } \calS_\fast(c^n)\to 0 \quad
 \Longrightarrow \quad c^n \to \Psi(\sfq). 
\end{equation} 
\end{lem}
\begin{proof} 
  Without loss of generality we may assume $c^n\to \ol c$. Hence we
  have $Q_\fast c^n \to Q_\fast \ol c=\sfq$.  Moreover, the continuity
  of $\calS_\fast$ gives
  $0=\lim \calS_\fast(c^n)=\calS_\fast(\ol c)$. Thus, we have
  $\ol c\in \scrEfa\cap \bfC_\sfq^\fast$. Now, the UFEC (see
  \eqref{eq:Mslow=graphPsi}) gives $\ol c = \Psi(\sfq)$ which is the
  desired result.
\end{proof}

We are now ready to establish the main compactness result. 

\begin{proof}[Proof of Theorem \ref{th:Compactness}.] \emph{Part (i):}  
From the energy bound $\cE(c^\eps(t)) \leq M_\text{ener}<\infty$
and the coercivity of $\cE$ we obtain an $\L^\infty$ bound for
$c^\eps$, namely
$0\leq c_j^\eps(t) \leq |c^\eps(t)|\leq \|c^\eps\|_{\L^\infty}\leq
M_\text{ener}$. 

\emph{Part (ii):} To provide a lower bound on $\cR_\eff$ we first
observe an upper bound on $\cR^*_\slow$, namely 
\[
\cR^*_\slow(c^\eps,  Q_\fast^\top \zeta) \leq \sum_{r\in R_\slow}
\kappa_r M_\text{ener}^{(\alpha^r+\beta^r)/2} \;
\sfC^*\Big(\gamma^r\cdot Q_\fast^\top \zeta \Big) \leq b_M \sfC^*\big(
b_Q |\zeta|\big)\quad \text{with }
b_Q = \max_{r\in R_\slow} |Q_\fast \gamma^r|,
\] 
where we used $0\leq c_j^\eps \leq M_\mafo{ener}$ from part (i). Using
the Legendre-Fenchel transformation and Proposition \ref{pr:EffDissPot} we obtain the lower bound
\[
\cR_\eps(c^\eps,v) \geq \wt\cR(c^\eps,Q_\fast v) \geq  \sup\bigset{ Q_\fast v\cdot  \zeta
   -b_M \sfC^*(b_Q |\zeta|) } {\zeta \in \R^{m_\fast} } 
\  = \ b_M \,\sfC\Big( \frac{|Q_\fast v|}{b_M b_Q} \Big).
\] 

Using the bound $M_\mafo{diss}$ for the dissipation functionals, the
family satisfies
\begin{align*}
\int_0^T \sfC\bigg( \frac{|Q_\fast \dot c^\eps(t)|}{b_Mb_Q}\bigg)  \dd t \leq
\int_0^T \frac1{b_M} \cR_\eps\big(c^\eps(t), \dot c^\eps(t)\big)  \dd t \leq 
\frac1{b_M} \fD_\eps(c^\eps) \leq M_\mafo{diss}/b_M.
\end{align*}
Since $\sfC(s) \geq \frac12|s|\log(1{+}|s|)$ for all $s \in \R$ (cf.\
\cite[Eqn.\,(A.2)]{MieSte19?CGED})  we have a uniform superlinear
bound for $Q_\fast \dot c^\eps$.  Thus, there exists
a subsequence (not relabeled) such that $Q_\fast \dot c^\eps \weak \mathsf w$ in
$\L^1([0,T];\R^{m_\fast})$. Moreover, $Q_\fast c^\eps$ is
equicontinuous (cf.\ \cite[Prop.\,5.9]{MieSte19?CGED}), which implies
$Q_\fast c^\eps \to \sfq^0$ in $\rmC^0([0,T];\Qspace)$. 

Because of $c^\eps \weak c^0$ we conclude $\sfq^0=Q_\fast c^0 \in
\W^{1,1}([0,T];\Qspace)$ and $\dot \sfq = \mathsf w$. Since the limit is
unique, we also know that the whole family converges. 

\emph{Part (iii):} The dissipation bound gives the estimate 
$\int_0^T \calS_\fast(c^\eps(t)) \dd t \leq \eps \,M_\mafo{diss}$.
Using $\calS_\fast(c)\geq 0$ this implies that
$f_\eps=\calS_\fast\circ c^\eps$ converges to $0$ in $\L^1([0,T])$.
Thus, we may choose a subsequence (not relabeled) such that $f_\eps(t)
\to 0$ a.e.\ in $[0,T]$. 

By the continuity $\calS_\fast$ and $|c^\eps(t)| \leq M_\mafo{ener}$ we also
know that $(f_\eps(t))_{\eps \in {]0,1[}}$ is bounded, while part (ii) provides
the convergence $Q_\fast c^\eps(t) \to \sfq^0(t)=Q_\fast c^0(t)$. Hence, Lemma
\ref{le:Cvg2Mslow} guarantees $c^\eps(t) \to \wt c(t):=\Psi(Q_\fast c^0(t))$
a.e.\ in $[0,T]$. By $c^\eps \weak c^0$ we have $c^0(t)=\wt c(t)$ a.e.
  
Since $\Psi$ is continuous by Proposition \ref{prop:InterEquil}, also
$\wt c=\Psi(Q_\fast c^0)$ is continuous.

\emph{Part (iv):} This follows via part (i), the pointwise a.e.\ convergence
established in the proof of part (iii), and from the 
dominated-convergence theorem.
\end{proof}


\subsection{Liminf estimate}
\label{su:LiminfEst}

The liminf estimate follows in a straightforward manner by using the fact that
the velocity part $\cR_\eps$ in $\fD_\eps$ satisfies the monotonicity
$\cR_\eps \geq \cR_\eff$, see Proposition \ref{pr:EffDissPot}, and
that the slope part $\calS_\eps $ takes the simple form
$\calS_\slow + \frac1\eps \calS_\fast$.

\begin{thm}[Liminf estimate]\label{th:LiminfEst}
Let $(c^{\eps})_{\eps>0}$ with $c^\eps \weak c^0$ in
$\L^1([0,T];\bfC)$ as in Theorem
\ref{th:Compactness} we have the estimate 
 $ \fD_0(c^0) \leq \liminf_{\eps \to 0^+}\fD_{\eps}(c^{\eps})$.
\end{thm}
\begin{proof}
  We may assume that
  $\alpha_{*}:=\liminf_{\eps \to 0}\fD_{\eps}(c^{\eps})<\infty$, since
  otherwise the desired estimate is trivially satisfied. This implies
  $\calS_\fast(c^0(t))=0$ a.e.\ in $[0,T]$ as in the previous proof.  We define the functional
\[
\mathfrak I(c,\sfq):= \int_0^T \calF(c(t),\sfq(t)) \dd t \quad \text{with }
\calF(c,\sfw)= \wt\cR(c,\sfw) + \calS_\slow(c).
\]
Then, using $\cR_\eps \geq \cR_\eff$ and $\calS_\eps\geq
\calS_\slow$, we have
\[
\fD_\eps(c^\eps) \geq \mathfrak I(c^\eps,Q_\fast \dot c^\eps) \quad
\text{and} \quad \fD_0(c^0)=\mathfrak I(c^0, Q_\fast \dot c^0),
\]
where the last identity follows from the construction of the density
$\calF$ via $\wt\cR$ and $\calS_\slow$, and
$\calS_0(c(t))=\calS_\slow(c(t))$ a.e.\ because of
$\calS_\fast(c^0(t))=0$.

Thus, it suffices to show the lower semicontinuity $\mathfrak I(c^0,Q_\fast \dot c^0)\leq
\liminf_{\eps\to 0^+} \mathfrak I(c^\eps,Q_\fast \dot c^\eps)$. Using
the strong convergence $c^\eps \to c^0$ in $\L^p([0,T];\bfC)$ and the
weak convergence $Q_\fast \dot c^\eps \weak Q_\fast \dot c^0$ in
$\L^1([0,T];\R^{m_\fast})$, see Theorem \ref{th:Compactness}(ii+iv),
this follows by Ioffe's theorem (cf.\ \cite[Thm.\,7.5]{FonLeo07MMCV}
if  $\calF:\bfC\ti \R^{m_\fast} \to [0,\infty]$ is lower
semicontinuous. However, the lower semicontinuity of $(c,\sfw) \mapsto
\calF(c,\sfw)=\wt\cR(c,\sfw) + \calS_\slow(c)$ follows immediately
from the continuity of $\calS_\slow$ and the by Legendre transforming
the continuous function $(c,\zeta) \mapsto \cR^*_\slow(c,Q_\fast^\top
\zeta)$. 

This finishes the proof of Theorem \ref{th:LiminfEst}.
\end{proof}

\subsection{Construction of the recovery sequence}
\label{su:Recovery}

In this section we construct the recovery sequence which
completes the proof of the Mosco convergence
$\fD_\eps \MoscoLimE \fD_0$ with energy constraint. Below in Step 1, we
will need the positivity and monotonicity condition \eqref{eq:InterEquilCond} for
$\theta \mapsto \Psi(\sfq {+}\theta \ol \sfq)$.

\begin{thm}[Limsup estimate]\label{th:Recovery}
  Let $c^{0} \in \L^1([0,T]; \bfC )$ with
  $\sup_{t\in[0,T]}\cE(c^{0}(t))<\infty$.  Then there exists a family
  $(c^{\eps})_{\eps\in]0,1]}$ with
  $\sup_{t\in[0,T],\ \eps\in]0,1]}\cE(c^\eps(t)) \leq M_\mafo{ener}
  <\infty$, \ $c^{\eps} \to c^{0}$ strongly in $\L^1([0,T]; \bfC)$,
  and $\lim_{\eps \to 0}\fD_{\eps}(c^{\eps})=\fD_{0}(c^{0})$.
\end{thm}
\begin{proof}
We prove the theorem in several steps. In Steps 1 and 2 we show that it is
sufficient to consider $c^0 \in \W^{1,\infty}([0,T];\bfC)$ with
$c_j^0(t)\geq \ul c>0 $, where we only work in $\fD_0$ which has the
advantage that $\calR_\eff(c,\dot c)$ only depends on $(\sfq,\dot\sfq) =
(Q_\fast c, Q_\fast \dot c)$, see Section \ref{su:CoarseGrain}. In Step 3 we
construct a recovery sequence, and in Step 4 we conclude with a
diagonal argument. 

\emph{Step 0:} To start with we may assume $\fD_0(c^0)<
\infty$. Indeed, if $\fD_0(c^0)=\infty$, then we choose $c^\eps=c^0$
and Theorem \ref{th:LiminfEst} gives $\liminf_{\eps\to 0}
\fD_\eps(c^\eps) \geq \fD_0(c^0)=\infty$, which means
$\fD_\eps(c^\eps) \to \infty$ as desired. 

\emph{Step 1. Reducing to positive curves $c^0$:} For $c^0$ with
$\fD_0(c^0)<\infty$ we know that $Q_\fast c^0 \in \W^{1,1}([0,T];\Qspace)$ and
$c^0\in \rmC^0([0,T];\bfC)$ after choosing the continuous
representative $c^0=\wt c$, see 
Theorem \ref{th:Compactness}. Exploiting the positivity and
monotonicity condition \eqref{eq:InterEquilCond} we now set
\[
  \ul c^l(t) := \Psi\big( \sfq(t) + \theta_l \, \ol\sfq \big)  \quad
\text{with }\ \theta_l = \frac1{l{+}1}\in {]0,1[}\quad
  \text{for all }t\in [0,T]. 
\]
By this condition, we know that $\ul c^l(t)$ lies in $\bfC_+$ for all
$t\in [0,T]$, such that the continuity of $\ul c^l$ guarantees that
for each $l$ there exists a $\delta_l>0$ such that  
$\ul c_i^l(t)\geq \delta_l$ for all $i\in I$ and $t\in [0,T]$. 
 
By the continuity of $\Psi$ we have $\ul c^l \to c^0$ uniformly and
hence strongly in $\L^1([0,T];\bfC)$.  We now show 
\begin{equation}
  \label{eq:fD.ck.c0}
  \fD_0(\ul c^l)= \int_0^T \!\big\{ \cR_\eff(\ul c^l(t), 
 \dot{\ul c}{}^l(t)) + \calS_0(\ul c^l(t))\big\} \dd t \ \to \ \fD_0(c^0)\quad \mathrm{as}\quad l\to 0. 
\end{equation}
For the second part, we use $\ul c^l(t)\in \scrMsl$ by construction
via $\Psi$, and the continuity of $\calS_\slow$ yields 
$\calS_0(\ul c^l(t)) = \calS_\slow( \ul c^l(t))
\to \calS_\slow(c^0(t)) = \calS_0(c^0(t))$ uniformly in $[0,T]$. 

For the first part we use (i) the special form of $\cR_\eff$ derived
in Proposition \ref{pr:EffDissPot}, namely
$\cR_\eff(c, v)= \wt\cR(c,Q_\fast v)$, where $\wt\cR(c,\cdot)$ is the
Legendre transform of $ \cR^*_\slow(c, Q^\top \,\cdot\,)$. Moreover,
the cosh-type dual dissipation potential $\cR^*_\slow $ as defined in
\eqref{eq:RRE.GS.b} or \eqref{eq:RRE.fs} enjoys (ii) a monotonicity
property namely $\cR^*_\slow( c,\xi) \leq \cR^*_\slow(\wt c, \xi)$ or
equivalently $\cR_\slow( c,v) \geq \cR_\slow(\wt c, v)$ if
$ c \leq \wt c$ componentwise. This can be exploited because of
the monotonicity condition \eqref{eq:InterEquilCond} using
$\ul c^l(t)\geq c^0(t)$ componentwise. With
$Q_\fast \,\ul{\dot c}{}^l(t)= \dot \sfq(t)$ for all $l\in \N$ we obtain
\begin{align*}
  & \int_0^T \! \cR_\eff(\ul c^l, \dot{\ul c}{}^l(t))
\dd t \ \overset{\text{(i)}} = \  \int_0^T \!
  \wt\cR( \ul c^l,  \dot \sfq(t)) \dd t 
\ \overset{\text{(ii)}}\longrightarrow \  \int_0^T \!
  \wt\cR(c^0, \dot \sfq(t)) \dd t 
\ \overset{\text{(i)}}= \ \int_0^T \!
  \cR_\eff(c^0,\dot c^0) \dd t , 
\end{align*}
where the convergence $\overset{\text{(ii)}}\longrightarrow$ follows from the
dominated-convergence theorem, since  the integrands on the left-hand side
are bounded by that on the right-hand side and we have pointwise
convergence.  With this we have established the desired convergence
\eqref{eq:fD.ck.c0}.  \smallskip 

\emph{Step 2. Reducing to bounded derivative $\dot\sfq=Q_\fast \dot
  c$:} Because of Step 1, we can now assume 
\[
c^0(t)\in \bfC_\delta:=
\bigset{c \in \bfC}{ |c|\leq 1/\delta, c_i\geq \delta \text{ for all }
i\in I} \ \text{ for all } t\in [0,T]
\]
where $\delta>0$. Moreover, as in \cite[Step 2(b) of proof of
Thm.\,5.12]{MieSte19?CGED} we find $\Lambda^*$ such that 
\[
c,\wt c \in \bfC_\delta \text{ and } |c{-}\wt c|\leq \alpha
<\frac1{2\Lambda^*}\quad  \Longrightarrow \quad   
\wt\cR(\wt c, \sfw) \ \leq \  (1{+}\Lambda^* \alpha)\,  \wt\cR(c,\sfw).
\] 
 With this we can estimate $\cR^*_\slow(c,\cdot)$
from below and hence $\cR_\eff$ from above. Moreover, we can use the
Lipschitz continuity of $c \mapsto \cR^*_\eps $.

For $\sfq(t)=Q_\fast c^0\in \W^{1,1}([0,T];\Qspace)$ we define the
piecewise affine interpolants  $\wh\sfq^k$ via  
\[
\wh\sfq^k\big((n{+}\theta) 2^{-k}T\big) = (1{-}\theta) \sfq\big(n2^{-k}T\big) +
\theta \sfq \big((n{+}1)2^{-k}T\big) \text{ for } \theta\in [0,1], \
n\in \{0,\ldots,2^k{-}1\}
\]
and the piecewise constant interpolant $\ol\sfq^k\big((n{+}\theta)
2^{-k}T\big)= \sfq(2^{-k} nT)$ for $\theta\in {[0,1[}$. We
also set $\wh c^{k}(t) = \Psi(\wh \sfq^k(t))$ and $\ol
c^k(t)=\Psi(\ol\sfq^k(t)) $. By standard
arguments we have 
\[
\|\ol c^k - \wh c^k\|_{\L^\infty} + \| \wh c^k - c^0\|_{\L^\infty} 
=: \alpha_k \to 0 \quad \text{for } k\to \infty.
\]

As in Step 1 we again find $\int_0^T \calS_0(\wh c^k(t)) \dd t \to
\int_0^T \calS_0(c^0(t))\dd t$. To treat the velocity part we use both
interpolants obtain the estimate 
\begin{align*}
&\int_0^T\! \cR_\eff(\wh c^k,\dot{\wh c}^k)\dd t \ 
= \ \int_0^T\! \wt\cR(\wh c^k,\dot{\wh\sfq}{}^k)\dd t  
\ \leq \  (1{+}\Lambda^*\alpha_k) \int_0^T\! \wt\cR(\ol c^k ,
\dot{\wh\sfq}^k) \dd t \\
&\ \overset{(\rmJ)}\leq \ (1{+}\Lambda^*\alpha_k) \int_0^T\! \wt\cR(\ol
c^k , \dot\sfq) \dd t  \ \leq \   
\int_0^T\! \wt\cR( c^0 , \dot\sfq) \dd t \ = \ 
(1{+}\Lambda^*\alpha_k)^2 \int_0^T\! \cR_\eff(c^0, \dot c^0) \dd t, 
\end{align*}  
where $\overset{(\rmJ)}\leq $ indicates the use of Jensen's
inequality applied to the convex integrand $\wt\cR(\ol c^k(t) , \,\cdot\,)$,
which is independent of $t$ in the intervals ${]2^{-k}nT,
  2^{-k}(n{+}1)T[}$.    
Combining this with the slope part and using $\alpha_k\to 0$ we obtain
the desired estimate $\limsup_{k\to \infty} \fD_0( \wh c^k) \leq
\fD_0(c^0)$, which is of course a limit because of the liminf estimate
in Theorem \ref{th:LiminfEst}.\smallskip

\emph{Step 3. The limsup for $\eps\to 0^+$:} By Steps 1 and 2 it is
sufficient to consider $ c^0 \in \W^{1,\infty}([0,T];\bfC)$ with
$ c^0(t)= \Psi(\sfq(t)) \in \bfC_\delta$ for some $\delta>0$. For these functions we
can now use the constant recovery sequence $c^\eps = c^0$, i.e.\ we will show
\begin{equation}
  \label{eq:fDeps.fD0.c0}
  \fD_\eps( c^0)=\int_0^T\! \big\{\cR_\eps(c^0,\dot c^0) +
\calS_\eps(c^0) \big\}\dd t \ \to \ \fD_0(c^0) = \int_0^T\! \big\{\cR_\eff(c^0,\dot c^0)
+ \calS_0(c^0) \big\} \dd t 
\end{equation}
for $\eps\to 0^+ $. Because of $c^0(t)\in \scrMsl$ we have $\calS_\eps(c^0(t)) =
\calS_\slow(c^0(t)) =  \calS_0(c^0(t))$, so the second summand of the
integral $\fD_\eps(c^0)$  converges trivially. 

Recall that $\Gamma= \mathOP{span}\bigset{\gamma^r}{r\in
  R=R_\slow\,\dot\cup\,R_\fast}$ and define a projection $\bbQ$ on
$\R^{i_*}$ with $\mathOP{im} \bbQ = 
\Gamma$ giving $\mathOP{ker} \bbQ^\top = \Gamma^\perp$. With this
we can estimate the dual dissipation potential $\cR_\eps^*$ from
below:
\[
\cR^*_\eps(c,\xi) \ \geq \ \cR_1^*(c,\xi) \geq b_*
|\bbQ^\top \xi|^2.
\]
To see this use $\sfC^*(\sigma) \geq \frac12 \sigma^2$ and
$\big(c^{\alpha^r}c^{\beta^r} \big)^{1/2} \geq
\delta^{(\alpha^r+\beta^r)/2} $ for all  $r\in R$. 

By Legendre-Fenchel transformation we obtain an upper bound for
$\cR_\eps$, where we use $\dot c^0 \in \Gamma$, i.e.\ $\bbQ \dot
c^0(t)=\dot c^0(t)$ (cf.\ Lemma \ref{le:ConservQuanti}): 
\[
\cR_\eps( c^0(t), \dot c^0(t))  \ \leq \ \cR_1( c^0(t), \dot c^0(t))
\ \leq \ \frac1{4b_*} |\bbQ \dot c^0(t)|^2 
 \ =  \ \frac1{4b_*} |\dot c^0(t)|^2.
\]
From $c^0 \in \W^{1,\infty}([0,T];\bfC)$ we see that $t \mapsto
\cR_1( c^0(t), \dot c^0(t))$ lies in $\L^\infty([0,T])$ and thus
provides an integrable majorant for $t \mapsto \cR_\eps( c^0(t), \dot
c^0(t))$. However, the convergence 
$\cR_\eps^*(c,\xi) \nearrow \cR_0^\eps = \cR^*_\slow +
\chi_{\Gamma_\fast^\perp}$ for $\eps \to 0^+$ implies 
$\cR_\eps(c,v) \searrow  \cR_\eff(c,v)$ for
all $(c,v)\in \bfC_\delta \ti \R^{i_*}$. Hence, Lebesgue's dominated
convergence theorem gives 
\[
\int_0^T \cR_\eps(c^0(t),\dot c^0(t)) \dd t \ \to \ \int_0^T
\cR_\eff(c^0(t),\dot c^0(t)) \dd t \quad \text{for } \eps \to 0^+,
\]
and \eqref{eq:fDeps.fD0.c0} is established.

\emph{Step 4. Diagonal sequence:} The full recovery sequence for a
general $c^0$ with $\fD_0(c^0)<\infty$ is obtained via
$\sfq(t)=Q_\fast c^0(t)$ as a diagonal
sequence $c^\eps = \Psi\big( \wh\sfq{}^{k(\eps)}(t) + \theta_{l(\eps)}
\ol\sfq\big)$, where the functions $k(\eps)$ and $l(\eps)$ are
suitably chosen such that $c^\eps \to c^0$ strongly in
$\L^1([0,T];\bfC)$ and $\fD_\eps(c^\eps) \to \fD_0(c^0)$.  It is also
clear from the construction that $\| c^\eps\|_{\L^\infty} \leq 1 +
\|c^0\|_{\L^\infty}$ such that the uniform energy bound
$\cE(c^\eps(t))\leq M_\text{ener}$ holds. 
\end{proof}

\paragraph*{Acknowledgments.} The research was partially supported by Deutsche
Forschungsgemeinschaft (DFG) through the Collaborative Research Center SFB 1114
``\emph{Scaling Cascades in Complex Systems}'' (Project no.\ 235221301),
Subproject C05 ``Effective models for materials and interfaces with multiple
scales''.  The authors are grateful to Michiel Renger for helpful and
stimulating discussions. 

\footnotesize


\newcommand{\etalchar}[1]{$^{#1}$}
\def\cprime{$'$}

\end{document}